\documentclass[journal,twoside,web]{ieeecolor}

\usepackage{etoolbox}
\makeatletter
\@ifundefined{color@begingroup}%
  {\let\color@begingroup\relax
   \let\color@endgroup\relax}{}%
\def\fix@ieeecolor@hbox#1{%
  \hbox{\color@begingroup#1\color@endgroup}}
\patchcmd\@makecaption{\hbox}{\fix@ieeecolor@hbox}{}{\FAILED}
\patchcmd\@makecaption{\hbox}{\fix@ieeecolor@hbox}{}{\FAILED}
\usepackage{generic}
\usepackage{cite}
\usepackage{textcomp}
\usepackage{graphicx}      
  
\usepackage{amsmath,amssymb}
\usepackage{amsthm}
\usepackage[mathscr]{eucal}
 
\usepackage{bbm}
\usepackage{bm}
\usepackage{soul}
\usepackage{color}
\usepackage{standalone}
\usepackage{dsfont}
\newtheorem{Theorem}{\bf Theorem}
 
\newtheorem{Lemma}{\bf Lemma}
\newtheorem{Remark}{\bf Remark}

\newtheorem{Assumption}{\bf Assumption}

\newcommand{\bs}[1]{\boldsymbol{#1}}%

\newcommand\scalemath[2]{\scalebox{#1}{\mbox{\ensuremath{\displaystyle #2}}}}

\newcommand{\Prob}{\mathbb{P}}

\DeclareMathOperator{\blkdiag}{blkdiag}

\newcommand{\snorm}[1]{\ensuremath{\| #1\|}}

\newcommand{\bsw}{\boldsymbol{w}}

\newcommand{\Pemp}{\bar{\mathbb{P}}_{\pi}^{\hat{\mathcal{M}}}}
\newcommand{\Pcl}{{\mathbb{P}}_{\pi}^{\mathcal{M}}}
\newcommand{\Pclemp}{\bar{\mathbb{P}}_{\pi}^{\mathcal{M}}}
\newcommand{\Pwhat}{\bar{\mathbb{P}}_{\hat{\bsw}}}
\newcommand{\Pwbar}{\bar{\mathbb{P}}_{\bsw}}
\newcommand{\Pw}{\mathbb{P}_{\bsw}}

\newcommand{\wemp}{\hat{\bs{w}}^{i}}
\newcommand{\wtremp}{{\bs{w}}^{i}}

\newcommand{\yemp}{\hat{\bs{y}}^{i}}
\newcommand{\yclemp}{\bs{y}_{\mathrm{cl}}^{i}}

\newcommand{\PHI}{\Phi}
\newcommand{\PHIh}{\PHI}
\newcommand{\PHItilde}{\tilde{\Phi}}

\newcommand{\DELTAtilde}{\PHIh Z \Delta\!\mathcal{M}}
\newcommand{\xu}{\begin{bmatrix}\bs{x}\\ \bs{u} \end{bmatrix}}
\newcommand{\xui}{\begin{bmatrix}\bs{x}^i\\ \bs{u}^i \end{bmatrix}}

\newcommand{\hatK}{\mathcal{K}}
\newcommand{\resolvent}{R_{\Phi}}
\newcommand{\triangleqq}{:=}

\usepackage{soul}

\def\BibTeX{{\rm B\kern-.05em{\sc i\kern-.025em b}\kern-.08em
    T\kern-.1667em\lower.7ex\hbox{E}\kern-.125emX}}
\begin{document}
\author{Francesco Micheli, Anastasios Tsiamis and John Lygeros
\thanks{Research supported by the European Research Council under the H2020 Advanced Grant no. 787845 (OCAL).\\
F. Micheli, A. Tsiamis and J. Lygeros are with the Automatic Control Laboratory in the Department of Information Technology and Electrical Engineering, ETH Z\"{u}rich, Switzerland. Emails: \texttt{\{frmicheli, atsiamis, jlygeros\}@ethz.ch}. }
}

\title{Data-Driven Distributionally Robust System Level Synthesis} 
\maketitle
\begin{abstract}
We present a novel approach for the control of uncertain, linear time-invariant systems, which are perturbed by potentially unbounded, additive disturbances.
We propose a \emph{doubly robust} data-driven state-feedback controller to ensure reliable performance against both model mismatch and disturbance distribution uncertainty. 
Our controller, which leverages the System Level Synthesis parameterization, is designed as the solution to a distributionally robust finite-horizon optimal control problem. The goal is to minimize a cost function while satisfying constraints against the worst-case realization of the uncertainty, which is quantified using distributional ambiguity sets. The latter are defined as balls in the Wasserstein metric centered on the predictive empirical distribution computed from a set of collected trajectory data. 
By harnessing techniques from robust control and distributionally robust optimization, we characterize the distributional shift between the predictive and the actual closed-loop distributions, and highlight its dependency on the model mismatch and the uncertainty about the disturbance distribution.
We also provide bounds on the number of samples required to achieve a desired confidence level and propose a tractable approximate formulation for the doubly robust data-driven controller. To demonstrate the effectiveness of our approach, we present a numerical example showcasing the performance of the proposed algorithm.
\end{abstract}
\begin{IEEEkeywords}
Distributionally Robust Control, Robust Control, Uncertain Systems, Predictive Control for Linear Systems.
\end{IEEEkeywords}

	
\section{Introduction}

Dealing with uncertainty is a fundamental challenge in many control applications. Oftentimes, the dynamics of the system and the distribution of the disturbance acting on it are unknown and should be accounted for. Robust and stochastic approaches have been developed in the last two decades
to specifically address both types of uncertainty~\cite{kouvaritakis2015developments}. Robust methods~\cite{bemporad2007robust,mayne2000constrained} assume bounded uncertainties and solve a worst-case optimization problem to provide guarantees against any possible realization of the uncertainty. Formulations have been developed to account for uncertainties in both the model and the realization of the disturbance~\cite{munoz2015robust}. However, since robust approaches account for \textit{all} possible realizations of the uncertainties, they neglect any available distributional information, leading to conservative control policies.

Stochastic methods~\cite{mesbah2016stochastic} can reduce this conservatism by imposing constraints that must be satisfied with a certain probability. However, analytical solutions in the stochastic setting can be obtained only under specific assumptions about the distribution of the uncertainty~\cite{cannon2009probabilistic,cannon2010stochastic}.
Alternatively, randomized methods such as the sample average approximation~\cite{kleywegt2002sample} and the scenario approach~\cite{calafiore2006scenario} can be used to reformulate the stochastic problem into large, but finite-dimensional, deterministic optimization problems. These methods can handle generic distributions and can be applied in the presence of uncertainty in both the dynamics and in the disturbance, as long as these distributions are accessible through sampling~\cite{micheli2022scenario,calafiore2012robust}. However, sampling-based approaches require a large amount of data to provide tight probabilistic guarantees. 

Following recent developements in the Distributionally Robust (DR) optimization literature \cite{gao2023distributionally, blanchet2022optimal, delage2010distributionally}, controllers based on DR formulations have been designed to blend the properties of robust and stochastic approaches. Similarly to the stochastic approach, the realizations of the disturbance come from a distribution. However, the distribution is uncertain and is allowed to vary within an ambiguity set; the goal is to optimize the controller performance against the worst-case distribution.
The DR approach has been applied in the control setting mainly to provide robustness against additive uncertainties~\cite{van2015distributionally,coppens2021data,li2021distributionally,taskesen2024distributionally,mark2020stochastic,zhong2021data,zolanvari2021data,aolaritei2023wasserstein,mcallister2023distributionally,hakobyan2024wasserstein,brouillon2023distributionally}. Typically, current approaches require having access to the \emph{true} model of the systems' dynamics.

Departing from prior work, we address the situation where an \emph{approximate model} of the dynamics is given, e.g., it can be obtained from some identification procedure. We are also provided with a limited amount of historical input-state \emph{trajectory data}, which are collected based on possibly closed-loop experiments, and can be used to characterize the disturbance distribution. 
A major challenge in this setting is that the model mismatch i) leads to erroneous predictions of the system evolution, and ii) makes designing a feedback controller harder. In addition, it also iii) affects our ability to recover the true disturbances from input-state trajectory data. On top of that, we only have access to a finite number of data.
These factors will inevitably induce a significant \emph{distribution shift} between the predicted and the actual closed-loop control performance. To address this issue,~\cite{micheli2022data} proposed an open-loop data-driven DR model predictive control formulation that robustly handles uncertainty in both the dynamics and the additive disturbance. Instead, here we study the closed-loop finite-horizon setting with state feedback and probabilistic state and input constraints. Note that the closed-loop setting is more challenging since it induces additional distribution shifts due to the model uncertainty, which is a well known problem in data-driven control~\cite{gevers2005identification}.

The contributions of this paper are the following.

\noindent
\textbf{Closed-loop DR formulation:} We propose a new distributionally robust data-driven state-feedback controller that is robust against model mismatches and uncertainty in the disturbance distribution. Given a set of input-state data collected from the system and a nominal model of the dynamics, we build the empirical predictive closed-loop distribution of states and inputs for the class of state-feedback controllers. Using the Wasserstein metric, we define an ambiguity set centered around the empirical predictive distribution. Utilizing tools from robust System Level Synthesis (SLS) and DR optimization, we pose the controller design problem as a stochastic optimization problem with respect to the worst-case probability distribution within the ambiguity set. 
Unlike typical DR optimization settings, where the disturbance is unaffected by the decisions, the presence of feedback changes the statistics of the closed-loop input and state distribution. As a result, we need to allow both the center and the radius of the ambiguity set to depend on the decision variables, i.e., the SLS parameters.

\noindent
\textbf{Distribution shift characterization:} We characterize the distributional shift between the predictive and the actual closed-loop distributions by upper bounding their Wasserstein distance. Hence, by carefully selecting the radius of the ambiguity set, we guarantee that the DR controller is robust against the actual closed-loop distribution with a prescribed confidence level.

\noindent
\textbf{Doubly robust solution:} Using robust SLS and DR optimization techniques we derive a tractable Linear Programming formulation for the DR optimization problem for piece-wise affine cost and constraint functions. We name it \emph{doubly robust} to highlight its ability to handle model mismatches and small sample sizes. To demonstrate the effectiveness of our approach, we present a numerical example showcasing the performance of the proposed controller.

The rest of the paper is organized as follows. In Section~\ref{Sec:Problem_Formulation} we define the problem setting and introduce the SLS formalism. In Section~\ref{Sec:SAAproblem} we derive the Sample Average Approximation and in Section~\ref{Sec:DR_Formulation} we formally state the distributionally robust control problem. Section~\ref{Sec:Shift} analyzes the distributional shift and Section~\ref{Sec:Reformulation} derives a tractable problem formulation for the class of piece-wise linear convex cost and constraint functions. Section~\ref{Sec:Extensions} describes how to extend the proposed framework to handle arbitrary initial conditions and an affine SLS parametrization. In Section~\ref{Sec:Numerics} we provide a numerical example that showcases the effectiveness of the proposed algorithm in a range of scenarios. Section~\ref{Sec:Conclusion} concludes the paper. 

\subsection{Further related work}
Distributionally robust control under known dynamics has been studied extensively. Van Parys et al. \cite{van2015distributionally} address the control of constrained stochastic linear systems with additive uncertainty under DR chance- and CVaR-constraints with second-order moment specifications. 
The authors in \cite{coppens2021data,li2021distributionally} tackle DR model predictive control formulations under moment-based ambiguity sets for the additive disturbance.
Taskesen et al. \cite{taskesen2024distributionally} address distributionally robust linear quadratic control with unknown noise distributions within Wasserstein ambiguity sets, as a generalization of the classical Linear-Quadratic-Gaussian control problem. Data-driven DR model predictive control formulation with Wasserstein ambiguity sets has been analyzed in \cite{mark2020stochastic,zhong2021data,zolanvari2021data,aolaritei2023wasserstein}. 
McAllister and Esfahani \cite{mcallister2023distributionally} show how the DR model predictive control formulation can recover important closed-loop properties of both robust and stochastic approaches. Hakobyan and Yang \cite{hakobyan2024wasserstein} tackle the partially observable case proposing a characterization of the Wasserstein ambiguity set based on the Gelbrich bound of the Wasserstein distance.

In the case of unknown dynamics, data-driven formulations typically use the data to account for unknown dynamics~\cite{berberich2020data,coulson2021distributionally, pan2023distributionally,yin2021maximum}. The disturbance can be assumed come from a distribution~\cite{coulson2021distributionally, pan2023distributionally,yin2021maximum,furieri2022near} or worst-case~\cite{berberich2020data}.
In~\cite{coulson2021distributionally,pan2023distributionally} data-driven DR formulations were considered. They either require full knowledge of the true disturbance distribution~\cite{pan2023distributionally} or they do not account for the effect of disturbance on the closed-loop trajectories~\cite{coulson2021distributionally}.

Our paper leverages tools from System Level Synthesis (SLS)~\cite{anderson2019system}, a convex parameterization for feedback design. This setting provides similar advantages to the disturbance affine feedback framework of Goulart et al.~\cite{goulart2006optimization} and the input-output parametrization of Furieri et al.~\cite{furieri2019input}. The SLS framework also allows for efficient robust control design under model uncertainty, e.g., using ideas from small-gain theory~\cite{anderson2019system,dean2020sample,chen2024robust}. Brouillon et al.~\cite{brouillon2023distributionally} consider a DR controller design using the SLS framework, but they require full knowledge of the true dynamics.

\subsection{Notation and preliminaries} 
We denote by $\delta_x$ the Dirac distribution at $x$. We denote by $\| \cdot\|$ the vector $\ell_1$-norm; for matrices the same symbol is used to denote the induced $\ell_1$-norm, i.e., $\|A\|=\max _{1 \leq j \leq n}\left(\sum_{i=1}^n\left|a_{i j}\right|\right)$. We define $(\cdot)_+ := \max\{\cdot,0\}$.
Let $z_0,z_1,\dots,z_t$ be any vector-valued sequence. 
By $z_{s:t}\triangleq \begin{bmatrix}
        z_{s}^\top &
        z_{s+1}^\top &
       \cdots &
        z_{t}^\top
    \end{bmatrix}^\top$ we denote the vector of stacked elements from index $s$ to index $t$.
Let $M_1,\dots,M_k$ be any sequence of matrices. Then, $\blkdiag(M_1,\dots,M_k)$ denotes the block diagonal matrix formed by $M_1,\dots,M_k$.

For a random variable $\omega\in\Omega\subseteq\mathbb{R}^r$ with distribution $\mathbb{P}_{\omega}$ and a function $\psi:\mathbb{R}^r\rightarrow\mathbb{R}$, the Conditional Value-at-Risk CVaR of level $\beta$ is defined as
\begin{equation*} 
    \text{CVaR}_{1-\beta}^{\omega\sim\mathbb{P}_{\omega}}\left(\psi(\omega)\right):= \inf_{t \in \mathbb{R}} \left[{\beta}^{-1}\, \mathbb{E}^{\omega\sim\mathbb{P}_{\omega}}\left[\left(\psi(\omega) +t\right)_{+}\right]-t\right]\ .
\end{equation*}
    
The Wasserstein metric~\cite{kantorovich1958space,mohajerin2018data} quantifies the minimum cost required to transform one distribution into another. 
    Consider distributions $\mathbb{Q}_1,\mathbb{Q}_2 \in \mathcal{M}(\mathcal{Y})$, where $\mathcal{M}(\mathcal{Y})$ is the set of all probability distributions $\mathbb{Q}$ supported on $\mathcal{Y}$ such that $\mathbb{E}^{\mathbb{Q}}\left[\| \bs{x} \|\right] < \infty $. The Wasserstein metric $d_{\mathrm{W}} : \mathcal{M}(\mathcal{Y}) \times \mathcal{M}(\mathcal{Y}) \rightarrow \mathbb{R}_{\geq 0}$ defines the distance between the distributions $\mathbb{Q}_1$ and $\mathbb{Q}_2$ as
    \begin{equation}\label{eq:Wasserstein_distance}
            d_{W}\left(\mathbb{Q}_{1}, \mathbb{Q}_{2}\right):=\inf_{\Pi} \left\{\int_{\mathcal{Y}^{2}}\left\|\bs{x}_{1}-\bs{x}_{2}\right\| \Pi\left(\mathrm{d} \bs{x}_{1}, \mathrm{d} \bs{x}_{2}\right)\right\},
    \end{equation}
    where $\Pi$ takes values in the set of joint distributions of $\bs{x}_1$ and $\bs{x}_2$ with marginals $\mathbb{Q}_1$ and $\mathbb{Q}_2$.

\section{Problem formulation}\label{Sec:Problem_Formulation}
Consider a discrete-time linear time-invariant (LTI) system
\begin{equation*}\label{eq:Dynamics}
    x_{k+1}={A} x_{k} + {B} u_{k} + w_{k}\ ,
\end{equation*}
with (fully measurable) state $x_k \in \mathbb{R}^{n}$, control input $u_k \in \mathbb{R}^{m}$.
The system is affected by the additive disturbance $w_k \in \mathbb{R}^{n}$ distributed according to some unknown probability distribution $\mathbb{P}_{w}$ defined over 
the unknown and possibly unbounded support set ${\mathcal{W}} \subseteq \mathbb{R}^{n}$. 
We are interested in a finite-horizon optimal control problem over some horizon $T$. For this reason,
we introduce the following batch notation
\begin{equation}\label{eq:dynamics_matrix_form}
    \bs{x}=Z\mathcal{A} \bs{x} + Z \mathcal{B} \bs{u} + \bs{w},
\end{equation}
where we concatenate the states, inputs, and disturbances of the system into stacked vectors
\begin{equation*}
\bs{x}\triangleqq
        x_{0:T}
 , \ \bs{u}\triangleqq
       u_{0:T}
, \ \bsw\triangleqq \begin{bmatrix}
        x_0 \\
      w_{0:T-1}
    \end{bmatrix}.
\end{equation*}
The batch system matrices are defined similarly
\begin{align*}
        \mathcal{A}\triangleqq \blkdiag\underbrace{(A,\dots,A)}_{T\!+\!1\text{ times}},\ \mathcal{B}\triangleqq \blkdiag\underbrace{(B,\dots,B)}_{T\!+\!1\text{ times}}.
\end{align*}
Finally, matrix $Z$ is the block-downshift operator: 
\begin{align*}
        Z\triangleqq \begin{bmatrix}
            0_{n\! \times\! n} &  &  &\\
            I_{n\! \times\! n} & \ddots & &\\
            &  \ddots& \ddots & \\
            &  &  I_{n\! \times\! n} & 0_{n\! \times\! n}
        \end{bmatrix},
\end{align*}
which maps $x_{0:T}$ to its delayed version $\begin{bmatrix}
    0& x^\top_{0:T-1}
\end{bmatrix}^\top$. It consists of $T\!+\!1 \times T\!+\!1$ blocks.

Our objective is to design a policy $\pi$ that minimizes a cost function while satisfying certain specifications.  
Here, we focus on causal linear state-feedback policies $\pi$ of the form 
\begin{equation*}
\pi:\:u_k = 
\sum_{t=0}^{k} K_{k,t} \ x_t\ , \quad k=0,\dots,T\ ,
\end{equation*}
or in vector form
\begin{equation}\label{eq:state_feedback_policy}
\bs{u} = \mathcal{K} \bs{x}\ ,
\end{equation}
where $\mathcal{K}$ is a block-triangular feedback matrix
\begin{equation*}
\mathcal{K}=\left[\begin{array}{cccc}
K_{0,0} & & & \\
K_{1,0} & K_{1,1} & & \\
\vdots & & \ddots & \\
K_{T,0} & \cdots & & K_{T,T}
\end{array}\right].
\end{equation*}
We could extend this setup to include an affine term. We can treat it similarly to the feedback term for the initial state, more details in Section~\ref{Sec:Extensions}.

Due to the stochastic nature of the state-input sequences, the control design is posed as a stochastic optimal control problem.
Define $\bs{y} \triangleqq \begin{bmatrix} \bs{x}^\top&\bs{u}^\top\end{bmatrix}^\top\in\mathbb{R}^{(T\!+\!1)n+Tm}$ and denote with $\Pcl$ the closed-loop distribution of $\bs{y}$ induced by the dynamics $\mathcal{M} \triangleqq [\mathcal{A}\ \mathcal{B}]$ under the policy $\pi$.

Now, we can formulate the following finite horizon stochastic optimal control problem.
\begin{equation}\label{eq:original_problem_y}
    \begin{aligned}
        {J}(\pi) \triangleqq \min_{\pi} \quad & \mathbb{E}^{\bs{y}\sim \Pcl} \left[ h\left(\bs{y}\right) \right]\\
        \text{s.t.}	\quad & \text{CVaR}_{1-\beta}^{\bs{y}\sim \Pcl} \left[g(\bs{y})\right] \leq 0 \\
    \end{aligned}\ ,
\end{equation}
where $h:\mathbb{R}^{(T\!+\!1)(m+n)}\rightarrow \mathbb{R}^+$
is the objective function, and $g:\mathbb{R}^{n+T(m+n)}\rightarrow \mathbb{R}$ the constraint function.

The distribution of the disturbance $\bs{w}$ and the system dynamics are both assumed to be uncertain, which makes solving~\eqref{eq:original_problem_y} challenging. Instead, we assume that we have access to trajectory data generated by system~\eqref{eq:dynamics_matrix_form}. 
\begin{Assumption}[Data collection]\label{Ass:data}
We have access to trajectory data in the form of a dataset $\mathcal{D}^{N,T\!+\!1}$, comprising $N$ independent $T \! + \! 1$-step state-input trajectories $\{
    \bs{x}^i,\bs{u}^i
\}$, $i=1,\dots,N$, that have been collected by applying inputs $\bs{u}^i$ of length $T$ to the system~\eqref{eq:dynamics_matrix_form}. The corresponding initial conditions and 
disturbances
are denoted by \[\bs{w}^i\triangleqq \begin{bmatrix}
    x^i_0\\w^i_{0:T-1}
\end{bmatrix}. \]
The initial condition is deterministic and fixed $x_0^i=x_0$, $i=1,\dots,N$ and the same for the trajectory collection phase and the controller deployment phase. 
\end{Assumption}
The assumption of fixed $x_0$ is for streamlining the presentation.
We can relax this requirement and allow $x_0$ to vary, see Section~\ref{Sec:Extensions}.
We further assume that the true system matrices $A$ and $B$ are unknown, but known to lie within a ball around some nominal system matrices $\hat{A}$ and $\hat{B}$ 
    \begin{equation}\label{eq:uncertainty_bounds_on_matrices}
    	\begin{aligned}
    		\|\Delta {A}\| \triangleqq \|\hat{{A}}-{A}\| \leq e_A\\
			\|\Delta {B}\| \triangleqq \|\hat{{B}}-{B}\| \leq e_B,
    	\end{aligned}
    \end{equation}
    for some $e_A, e_B>0$, where we recall that $\|\cdot \|$ denotes the $\ell_1$ induced norm. The nominal system dynamics $\hat{\mathcal{M}}=\begin{bmatrix}
    	\hat{\mathcal{A}}& \hat{\mathcal{B}}
    \end{bmatrix}$, and thus the model error $\Delta\mathcal{M}\triangleqq\begin{bmatrix}
    \Delta \mathcal{A}& \Delta \mathcal{B}
\end{bmatrix}$ have the same block diagonal structure as the true $\mathcal{M}$.
For clarity of exposition we are assuming here that this bound is deterministic.
\begin{Remark}[Role of data]
    In this paper, we use the state-input data to obtain (approximate) disturbance samples (see~\eqref{eq:nominal_empirical_distribution_of_the_disturbances}). To simplify the presentation, the nominal model $\hat{\mathcal{M}}$ and the bounds $\epsilon_A,\epsilon_B$ are assumed to be given a priori. However, we could use the same data to also perform system identification and obtain such a nominal estimate $\hat{\mathcal{M}}$~\cite{dean2020sample,simchowitz2018learning}.  In such a case, the uncertainty bounds~\eqref{eq:uncertainty_bounds_on_matrices} would hold in a probabilistic sense. The results of this paper can be extended to this setting.
\end{Remark}

\subsection{System Level Synthesis}\label{Subsec:SLS}
Under~\eqref{eq:state_feedback_policy}, the dynamics~\eqref{eq:dynamics_matrix_form} can be rewritten as
\begin{equation}\label{eq:closed-loop_responses}
\begin{aligned}
    \bs{y} &= \xu=
    \begin{bmatrix}
        \left( I - Z \left( \mathcal{A} + \mathcal{B}\mathcal{K} \right) \right)^{-1} \\
        \mathcal{K} \left( I - Z \left( \mathcal{A} + \mathcal{B}\mathcal{K} \right) \right)^{-1}
    \end{bmatrix} \bs{w}
\end{aligned}.
\end{equation}
Note that optimizing over the linear gains $\mathcal{K}$ in~\eqref{eq:original_problem_y} is a non-convex problem in general. 
To deal with this issue, we adopt the SLS framework.

Following the SLS formalism~\cite{anderson2019system}, we define the closed-loop \textit{system response matrices} $\Phi_x\in\mathbb{R}^{n(T\!+\!1)\times n(T\!+\!1)}$ and $\Phi_u\in\mathbb{R}^{m(T\!+\!1)\times n(T\!+\!1)}$, as the matrices that map the disturbance $\bsw$ to the state $\bs{x}$ and control inputs $\bs{u}$ respectively
\begin{equation}\label{eq:phi_dynamics}
    \bs{y} = \xu = \begin{bmatrix}\Phi_x\\\Phi_u\end{bmatrix} \bs{w}= \Phi \bs{w},
\end{equation}
where we define $\Phi\triangleqq\left[\begin{array}{cc}\Phi^\top_x&\Phi^\top_u\end{array}\right]^\top$.
By causality, both maps have block-triangular structure
\begin{equation}\label{eq:causality}
\begin{aligned}
    \Phi_x=\left[\begin{array}{cccc}
\Phi_x^{0,0} & & & \\
\Phi_x^{1,1} & \Phi_x^{1,0} & & \\
\vdots & \ddots & \ddots & \\
\Phi_x^{T, T} & \ldots & \Phi_x^{T, 1} & \Phi_x^{T, 0}
\end{array}\right], \\
\Phi_u=\left[\begin{array}{cccc}
\Phi_u^{0,0} & & & \\
\Phi_u^{1,1} & \Phi_u^{1,0} & & \\
\vdots & \ddots & \ddots & \\
\Phi_u^{T, T} & \ldots & \Phi_u^{T, 1} & \Phi_u^{T, 0}
\end{array}\right] .
\end{aligned}
\end{equation}
The core idea is to re-parameterize policy~\eqref{eq:state_feedback_policy} and perform the controller synthesis directly on the closed-loop system response matrices $\Phi$ that appear in~\eqref{eq:phi_dynamics}, instead of the state feedback map $\mathcal{K}$.
By identifying~\eqref{eq:closed-loop_responses} with~\eqref{eq:phi_dynamics}, the linear feedback parameterization is equivalent to the SLS parameterization  under the transformation
$$\mathcal{K} = \Phi_u \Phi_x^{-1} $$
when the following condition holds~\cite{anderson2019system}
\begin{equation}\label{eq:SLS_dynamics_constraint}
    \left[ I\!-\!Z\mathcal{A} \ -\!Z \mathcal{B}\right]\Phi = I\ .
\end{equation}

Note that the above constraint cannot be enforced explicitly as it requires knowledge of the true system matrices $A$, $B$, which are unknown. Since we only have access to the nominal system matrices $\hat{A}$, $\hat{B}$, we replace constraint~\eqref{eq:SLS_dynamics_constraint} with
\begin{equation}
 \label{eq:SLS_dynamics_constraint_nominal}
	\left[ I-Z\hat{\mathcal{A}} \ -Z\hat{\mathcal{B}}\right]\PHIh =I .
\end{equation}
We account for the model error in Section~\ref{Sec:DR_Formulation}.

As a result, we can reformulate the finite-horizon stochastic optimal control Problem~\eqref{eq:original_problem_y}  as:
\begin{equation}\label{eq:original_problem_phi}
    \begin{aligned}
	{J}(\pi)  \triangleqq \min_{\Phi} \quad & \mathbb{E}^{\bs{y}\sim\Pcl} \left[ h\left(\bs{y}\right) \right]\\
        \text{s.t.}	\quad & \eqref{eq:causality}, \ \eqref{eq:SLS_dynamics_constraint_nominal},  \\
	& \text{CVaR}_{1-\beta}^{\bs{y}\sim \Pcl} \left[g(\bs{y})\right] \leq 0 \\
	\end{aligned}\ .
\end{equation}

\subsection{Sample Average Approximation}\label{Sec:SAAproblem} 
Without any robustness considerations, an effective approach to solving problem~\eqref{eq:original_problem_phi} is to replace the expectations with the nominal empirical means. This is also known as Sample Average Approximation (SAA)~\cite{kleywegt2002sample}.
Given $N$ trajectory samples and the nominal model $\hat{\mathcal{M}}$, we can construct a nominal empirical distribution for the disturbances
\begin{equation}\label{eq:nominal_empirical_distribution_of_the_disturbances}
    \begin{aligned}
       & \Pwhat \triangleqq \frac{1}{N} \sum_{i=1}^{N} \delta_{\hat{\bs{w}}^{i}}\ ,\\
       &\hat{\bs{w}}^i = \bs{x}^i - Z \hat{\mathcal{A}} \bs{x}^i - Z \hat{\mathcal{B}} \bs{u}^i\ , \ \text{ for }\ i = 1,\dots, N,
    \end{aligned}
\end{equation}
where we use $\bar{\Prob}$ to denote empirical distributions. 

Having access to disturbance samples we can simulate the closed-loop performance of a state-feedback policy $\pi:\:\hatK=\Phi_u \Phi_x^{-1} $, using the nominal dynamics $\hat{\mathcal{M}}$. Define the \textit{empirical predictive} distribution of inputs and states as
\begin{equation}\label{eq:predictive_empirical_distribution}
\Pemp = \frac{1}{N} \sum_{i=1}^{N} \delta_{\hat{\bs{y}}^{i}},
\end{equation}
where, owing to~\eqref{eq:SLS_dynamics_constraint_nominal}, the $i-$th prediction is
\begin{equation}\label{eq:prediction_y_i}
    \hat{\bs{y}}^{i} \triangleqq \PHIh \hat{\bs{w}}^i, \quad i= 1,\dots, N\ .
\end{equation}
The superscript $\hat{\mathcal{M}}$ denotes that we are relying on the nominal model to construct the empirical distribution of the disturbances in~\eqref{eq:nominal_empirical_distribution_of_the_disturbances} and for the system forward simulation in~\eqref{eq:predictive_empirical_distribution}, while the subscript $\pi$ denotes the state-feedback policy~\eqref{eq:state_feedback_policy} induced by the feedback matrix $\mathcal{K}$.

Following the SAA approach, we optimize the empirical predicted performance, by taking the expectations in the cost and in the constraint with respect to the nominal empirical predictive distribution $\Pemp$.
\begin{equation}\label{eq:problem_SAA}
    \begin{aligned}
        {J}^{SAA}(\pi) \triangleqq \min_{\Phi} \quad & \mathbb{E}^{\bs{y}\sim \Pemp }\left[ h\left(\bs{y}\right) \right] \\
        \text{s.t.}	\quad & \eqref{eq:causality}, \ \eqref{eq:SLS_dynamics_constraint_nominal},  \\
        & \text{CVaR}_{1-\beta}^{\bs{y}\sim \Pemp} \left[ g \left( \bs{y} \right) \right] \leq 0 \\
    \end{aligned}\ .
\end{equation}
The SAA formulation has the advantage of being tractable and having strong asymptotic performance guarantees.
However, in the presence of model mismatch and when the number of samples $N$ is small, the SAA approach can overfit to the wrong model $(\hat{\mathcal{A}},\ \hat{\mathcal{B}})$ and the samples $\bs{x}^i,\,\bs{u}^i$, for $i\le N$, leading to optimistically biased solutions, which is referred to as the \textit{optimizer's curse} in the optimization literature~\cite{smith2006optimizer}. This can result in a large discrepancy between the in-sample predicted performance and the out-of-sample closed-loop performance. This \emph{distribution shift} is a consequence of the fact that, while approximating the true closed-loop distribution $\Pcl$ by computing the nominal empirical predictive distribution $\Pemp$ as in~\eqref{eq:predictive_empirical_distribution}, we are wrongfully assuming that i) the nominal model $\hat{\mathcal{M}}$ is an accurate representation of the true unknown model $\mathcal{M}$ and ii) the empirical distribution of the disturbance $\Pwhat$ obtained from $N$ samples is an accurate representation of the true disturbance distribution $\Pw$. 

\subsection{Distributionally robust formulation}\label{Sec:DR_Formulation}	
In practice, the nominal empirical predictive distribution $\Pemp$ will inevitably differ from the true closed-loop distribution $\Pcl$ due to model mismatch and the limited number of samples available. To account for this distribution shift, we follow a distributionally robust approach. We robustify problem~\eqref{eq:problem_SAA} against uncertainty in the predictive distribution, by optimizing over the worst-case expectation within a set of probability distributions, which we refer to as an \textit{ambiguity set}. 

In this paper, we consider ambiguity sets constructed using the Wasserstein metric. 
The Wasserstein metric is a popular choice for defining ambiguity sets as it can handle distributions with arbitrary supports, including finitely supported ones, making it computationally tractable for data-driven applications. Other types of ambiguity sets, such as those based on the Kullback-Leibler divergence or the total variation distance, may have drawbacks such as not being defined for distributions with different supports.
Using the Wasserstein metric, an ambiguity set of radius $\varepsilon>0$ around a probability measure $\Prob$ can be defined as
\begin{equation}
\mathcal{B}^{\varepsilon}\left(\Prob\right) \triangleqq \left\{\mathbb{Q}\in \mathcal{M}(\mathcal{Y})\middle| d_{W}\!\left(\Prob,\mathbb{Q}\right)\leq \varepsilon \right\},
\end{equation}
where the Wasserstein distance $d_{W}(\cdot,\cdot)$ is defined in~\eqref{eq:Wasserstein_distance}.
Here, we consider ambiguity sets  $\mathcal{B}^{\varepsilon}(\Pemp)$, centered around the empirical predictive distribution $\Pemp$ with radius $\varepsilon$. Typically, in DR optimization, the radius $\varepsilon$ is a constant and usually treated as a design parameter. In contrast, here we allow $\varepsilon$ to depend on the optimization variables, that is $\varepsilon\triangleqq\varepsilon(\PHI).$ Hence, both the ambiguity set center and its radius depend on the closed-loop responses $\PHI$. 
By appropriately choosing the function $\varepsilon(\Phi)$ we can ensure that the ambiguity set contains the true closed-loop distribution, i.e.
\begin{equation}\label{eq:closed_loop_distribution_included}
\Pcl\in \mathcal{B}^{\varepsilon}\left( \Pemp\right).
\end{equation} This would not be possible with a constant radius.

A DR version of the finite-horizon stochastic optimization problem~\eqref{eq:original_problem_phi} can now be written as
\begin{equation}\label{eq:DR_problem}
    \begin{aligned}
            {J}^{DR} \triangleqq& \min_{\PHI}  &&\sup_{\mathbb{Q} \in 
 \mathcal{B}^{\varepsilon(\PHI)}\left(\Pemp\right)} \mathbb{E}^{\bs{y}\sim\mathbb{Q}}\left[ h\left(\bs{y}\right) \right]\\
         &\text{s.t.}	\quad && \eqref{eq:causality}, \ \eqref{eq:SLS_dynamics_constraint_nominal},  \\
            &&&\sup_{\mathbb{Q} \in \mathcal{B}^{\varepsilon(\PHI)}\left(\Pemp\right)}\text{CVaR}_{1-\beta}^{\bs{y}\sim\mathbb{Q}}\left(g(\bs{y})\right) \leq 0, 
    \end{aligned}
\end{equation}

In Section~\ref{Sec:Shift}, we characterize the distribution shift $d_{W}(\Pemp, \Pcl)$ as a function of the optimization variable $\PHI$ and provide potential candidate functions for $\varepsilon(\PHI)$ so that~\eqref{eq:closed_loop_distribution_included} is satisfied. In this case, solving~\eqref{eq:DR_problem} will provide a control policy that is robust against all the distributions contained in the ambiguity set, including the true (unknown) closed-loop one.
Informally, this will allow us to claim that if the distributionally robust problem~\eqref{eq:DR_problem} is feasible and its minimizer $\pi^{DR}$, attains a cost $J^{DR}(\pi^{DR})$, then, with high confidence, $\pi^{DR}$ is a feasible solution for the original Problem~\eqref{eq:original_problem_y} and the resulting cost ${J}(\pi^{DR})$ is upper bounded by the computed $J^{DR}(\pi^{DR})$.

As is common in robust optimization, the worst-case in the cost and in the constraint are formulated independently, possibly introducing conservatism as the optimization problem optimizes against two separate worst-case distributions. 

To characterize the distribution shift, we require a technical assumption on the distribution of the multi-step disturbance vector $\bs{w}$.
\begin{Assumption}[Light-tail assumption]\label{Ass:Light_Tail}
    For some constants $\mathsf{a}>1$ and $\mathsf{b}>0$
    \begin{equation*}
            \mathscr{E}_{\mathsf{a},\mathsf{b}} = \mathbb{E}\left[e^{\mathsf{b} \|{\bs{w}}\|^\mathsf{a}}\right]< +\infty\ .
    \end{equation*}
\end{Assumption}
\noindent
This assumption is a condition on the decay rate of the tail of the probability distribution $\Pw$ and is satisfied when $\bs{w}$ is sub-Gaussian or when $\mathcal{W}$ is compact. It is required to obtain the finite-sample concentration bound of Lemma~\ref{Lemma:fournier} below.

Even with an accurate radius $\varepsilon$, we still need to solve~\eqref{eq:DR_problem}. This requires reformulating~\eqref{eq:DR_problem} in a way that makes the solution computationally practical. The main difficulty is that the ambiguity set depends on the optimization variables $\Phi_x$ and $\Phi_u$, as well as the model uncertainty $\Delta\!\mathcal{M}$. In Section~\ref{Sec:Reformulation}, we employ techniques inspired by robust control (small-gain theory) to obtain such a reformulation.


\section{Characterization of the distribution shift}\label{Sec:Shift}
In this section, we upper bound the Wasserstein distance $d_{W}\left(\Pemp, \Pcl\right)$ between the predictive empirical distribution and the actual closed-loop one. Since the empirical distribution is a random quantity, we can only provide an upper bound that holds with high probability. 

Let us first characterize the actual closed-loop distribution $\Pcl$ under policy~\eqref{eq:state_feedback_policy}. 
Note that the nominal system responses $\Phi$ in~\eqref{eq:problem_SAA},~\eqref{eq:DR_problem} satisfy the affine constraint~\eqref{eq:SLS_dynamics_constraint_nominal} for the inaccurate nominal dynamics $\mathcal{\hat M}$ instead of true model $\mathcal{M}$. 

Following the steps of Section~2.3 in~\cite{anderson2019system}, we can provide an exact expression for the effect of model mismatch. In particular, the state-feedback $\mathcal{K} = \Phi_u \Phi_x^{-1}$ induces the following state-input closed-loop distribution $\bs{y}_{\mathrm{cl}} \sim \Pcl$:
\begin{equation}\label{eq:closed-loop_distribution_actual}
		\bs{y}_{\mathrm{cl}} 
		= \left(I+\Phi Z\Delta\!\mathcal{M}\right)^{-1}\Phi \bs{w}= \resolvent \Phi \bs{w}\ ,\footnote{Note that by the inversion lemma, we can also write $R_\Phi \Phi=\Phi(I+Z\Delta\mathcal{M}\Phi)^{-1}$. The latter expression is the one derived in~\cite{anderson2019system}.}
\end{equation}
where $ \resolvent = \left(I+\Phi Z\Delta\!\mathcal{M}\right)^{-1}$. 
Due to the lower block-triangular structure (consequence of the causality requirement of the controller $\hatK$) the inverse always exists. 

Comparing~\eqref{eq:closed-loop_distribution_actual} to~\eqref{eq:predictive_empirical_distribution}, there are two sources of distribution shift: i) the model mismatch $\Delta\!\mathcal{M}$, and ii) the disturbance distribution uncertainty as we only have a finite number of samples. This distinction becomes transparent by leveraging the triangle inequality, leading to 
\begin{equation}\label{eq:Triangle_inequality}
    d_{W}\!\left(\Pemp\!, \Pcl\right) \!\leq \underbrace{d_{W}\!\left(\Pemp\!, \Pclemp\right)}_{\text{model mismatch}} \ + \underbrace{d_{W}\!\left(\Pclemp\!, \Pcl\right)}_{\substack{\text{disturbance distribution} \\ \text{uncertainty}}}\!,
\end{equation}
where $\Pclemp$ denotes the empirical (finite-sample) version of the true closed-loop distribution $\Pcl$,
\begin{equation}\label{eq:empirical_closed_loop_distribution}
	\Pclemp = \frac{1}{N} \sum_{i=1}^{N} \delta_{\bs{y}_{\mathrm{cl}}^{i}}\ ,
\end{equation}
with $$\yclemp \triangleqq \resolvent \Phi \wtremp\ $$ and 
\begin{equation}
 {\bs{w}}^i = \bs{x}^i - Z {\mathcal{A}} \bs{x}^i - Z {\mathcal{B}} \bs{u}^i\ ,
\end{equation}
is the true (unknown) disturbance that affected the sampling of the input-state trajectories $\bs{x}^i$, $\bs{u}^i$, of dataset $\mathcal{D}^{N,T}$.
We define $$ \Pwbar:= \frac{1}{N} \sum_{i=1}^{N} \delta_{{\bs{w}}^{i}}\ .$$ 
This distribution could have been obtained using the true (unknown) dynamics $\mathcal{M}$, however, since we do not know the true model $\mathcal{M}$, we do not have access to the empirical closed-loop distribution $\Pclemp$; we only use it as an intermediate quantity for controlling the distribution shift induced by the closed-loop policy.

\begin{Theorem}[Distribution Shift]\label{Thm:tri_ineq}
The distance between the predictive empirical distribution and the closed-loop distribution is upper-bounded by:
    \begin{equation*}
    	\scalemath{0.9}{\begin{aligned}
    		&d_{W}\!\left(\Pemp, \Pcl\right) \leq \\
    		&\frac{1}{N}\sum_{i=1}^N \left\| \resolvent \Phi Z \Delta\!\mathcal{M} \left(\PHIh \wemp -\xui \right) \right\|  + \left\| \resolvent \Phi Z \right\| d_{W}\!\left(\Pwbar, \Pw\right) .
    	\end{aligned}}
    \end{equation*}
\end{Theorem}
\begin{proof}
    The result follows directly by the triangle inequality~\eqref{eq:Triangle_inequality}, Lemma~\ref{Lem:1},  and Lemma~\ref{Lem:2}.
\end{proof}
Theorem~\ref{Thm:tri_ineq} decomposes the distribution shift between the empirical predictive distributions, $\Pemp$, and the true closed-loop distribution, $\Pcl$, into a \textit{model mismatch} and a \textit{disturbance distribution uncertainty} error component. The two technical lemmas bound these two components separately. 

The model mismatch term depends on the product of three terms: i) the closed-loop responses, ii) the model error $\Delta\!\mathcal{M}$, and iii) the distance between the predicted input-state trajectories, $\Phi_x\bs{\hat{w}}^i,\,\Phi_u\bs{\hat{w}}^i$, and the collected input-state trajectories $\bs{x}^i,\,\bs{u}^i $. It goes to zero as the model error approaches zero or it can also be made small if the planned input-state trajectories are close to the collected trajectories. 

\begin{Remark}
 If the data in Assumption~\ref{Ass:data} is collected using some (known or unknown) controller with closed-loop map ${\PHIh}^{\text{old}}$, as often required in safety-critical applications, we can write
\begin{equation*}
    \PHIh \wemp -\xui = \left(\PHIh - {\PHIh}^{\text{old}} \right) \wemp \ ,
\end{equation*}
which highlights that this component can be small when the new controller is close to the one used in the data collection. Note that the right-hand side expression arises implicitly via the data $\bs x^i,\bs u^i$ on the left-hand side. We do not need explicit access to a closed-loop map ${\PHIh}^{\text{old}}$.
This term could be useful in an episodic learning-based control setting~\cite{gevers2005identification},
where the controller is updated across episodes.
\end{Remark}

The second component captures the distributional shift due to the Wasserstein distance $d_{W}\!\left(\Pwbar,\Pw\right)$ between the true distribution and the finite-sample empirical distribution of the disturbance. This component persists under zero model error and goes to zero only if the empirical distribution $\Pwbar$ approaches the true one $\Pw$, i.e., as the number of samples $N$ goes to infinity.

Note that both components depend on the actual closed-loop responses $\resolvent \Phi$. As the model mismatch $\Delta\!\mathcal{M}$ gets smaller, the closed-loop responses get closer to the nominal ones. In the following, we upper-bound every component separately.

\begin{Lemma}[Model mismatch]\label{Lem:1}
The Wasserstein distance between the empirical predictive distribution $\Pemp$ and the empirical closed-loop distribution $\Pclemp$ is upper bounded as:
\begin{equation*}
    \begin{aligned}
            &d_{W}\!\left(\Pemp, \Pclemp \right)
            \leq \frac{1}{N}\sum_{i=1}^N \left\| \resolvent \Phi Z \Delta\!\mathcal{M} \left(\PHIh \wemp -\xui \right) \right\|
\end{aligned}
\end{equation*}
\end{Lemma}
\begin{Lemma}[Disturbance distribution uncertainty]\label{Lem:2}
Let $d_{W}\!\left(\Pwbar, \Pw \right)$ be the Wasserstein distance between the empirical disturbance distribution $\Pwbar$ and the true $\Pw$. The Wasserstein distance between the empirical closed-loop distribution $\Pclemp$ and the true closed-loop distribution $\Pcl$ can be upper bounded as:
\begin{equation*}\scalemath{1}{
		d_{W}\!\left(\Pclemp, \Pcl\right) \leq \left\| \resolvent \Phi Z \right\| d_{W}\!\left(\Pwbar, \Pw\right)}
	\end{equation*}
\end{Lemma}
Theorem~\ref{Thm:tri_ineq} gives us a way of setting the radius of the ambiguity set in~\eqref{eq:DR_problem}.
However, it requires knowledge of the Wasserstein distance $d_{W}\!\left(\Pwbar, \Pw\right)$. 
For this, we can use the following finite-sample convergence result~\cite{fournier2015rate}.

\begin{Lemma}{(\cite[Theorem 2]{fournier2015rate})}\label{Lemma:fournier}: Under Assumption~\ref{Ass:Light_Tail}, for $nT > 2$, for all $\kappa>0$, $ N \in \mathbb{N}$,
\begin{equation*}\scalemath{0.97}{
            \mathbb{P}^N \left\{d_{W}\!\left(\Pwbar, \Pw\right)\geq \kappa\right\} \leq
            \left\{\begin{aligned}
                &\mathsf{c}_1 \exp \left(-\mathsf{c}_2 N \kappa^{nT}\right) &\text{if } \kappa \leq 1 \\
                &\mathsf{c}_1 \exp \left(-\mathsf{c}_2 N \kappa^{\mathsf{a}}\right) &\text{if } \kappa>1\
            \end{aligned} \right. , }
    \end{equation*}
    where $\mathbb{P}^N$ is the $N$-fold distribution of the disturbance generating process. The positive constants $\mathsf{c}_1$ and $\mathsf{c}_2$ depend on the dimensions of $\bsw$ and on the constants $\mathsf{a}$, $\mathsf{b}$ and $\mathscr{E}_{\mathsf{a},\mathsf{b}}$ from Assumption~\ref{Ass:Light_Tail}.
\end{Lemma}
\noindent
Combining Theorem~\ref{Thm:tri_ineq} and Lemma~\ref{Lemma:fournier} leads to the final probabilistic bound.
\begin{Theorem}[Finite-sample Guarantees]\label{Thm:worst-case_guarantees}
Fix a failure probability $\eta>0$. Under Assumptions~\ref{Ass:Light_Tail}, select the (decision dependent) radius $\varepsilon=\varepsilon(\Phi)$ of the ambiguity set $\mathcal{B}^{\varepsilon}\left(\Pemp\right)$ such that
\begin{equation*}
    \varepsilon(\Phi) \! \geq \! \frac{1}{N}\!\!\sum_{i=1}^N \!\left\| \! \resolvent \Phi Z \Delta\!\mathcal{M} \!\left(\!\PHIh \!\wemp \!-\!\xui \right)\! \right\| + \left\| \resolvent \Phi Z \right\| \kappa(\eta, N),
\end{equation*}
with
\begin{equation}\label{eq:kappa_confidence}
    \begin{aligned}
        \kappa\!\left(\eta, N\right)\!\triangleqq \!\left\{ \! \begin{aligned}
            \!&\left(\frac{\log \left(\mathsf{c}_1/ \eta\right)}{\mathsf{c}_2 N}\right)^{\frac{1}{nT}} \text {if } \!N\! \geq \! \frac{\log\! \left(\mathsf{c}_1/ \eta\right)}{\mathsf{c}_2 N} \\
            \!&\left(\frac{\log \left(\mathsf{c}_1/ \eta\right)}{\mathsf{c}_2 N}\right)^{\frac{1}{\mathsf{a}}} \text {if } \!N\! <\! \frac{\log\! \left(\mathsf{c}_1/ \eta\right)}{\mathsf{c}_2 N}
        \end{aligned}\right. \!,
    \end{aligned}
\end{equation}
and $\mathsf{c}_1$, $\mathsf{c}_2$ positive constants that depend on the dimensions of $\bsw$ and on the constants $\mathsf{a}$, $\mathsf{b}$ and $\mathscr{E}_{\mathsf{a},\mathsf{b}}$ from Assumption~\ref{Ass:Light_Tail}. 
Let $\PHItilde=[\tilde{\Phi}_x^\top, \tilde{\Phi}_u^\top]^\top$ be feasible for Problem~\eqref{eq:DR_problem}, and let $\tilde{\pi}$ be the state-feedback policy induced by $\tilde{\mathcal{K}} = \tilde{\Phi}_u \tilde{\Phi}_x^{-1}$, then
	\begin{equation*}
		\left\{\begin{aligned}
			\ &\mathbb{P}^N\left\{{J}\left(\tilde{\pi} \right) \leq \hat{J}^{DR}\left(\tilde{\pi}\right)\right\} \geq 1-\eta \\
			&\mathbb{P}^N\left\{\text{CVaR}_{1-\beta}^{\bs{y}\sim \mathbb{P}^{\mathcal{M}}_{\tilde{\pi}}}\left(g(\bs{y})\right) \leq 0 \right\} \geq 1-\eta
		\end{aligned}\right.\ .
	\end{equation*}
\end{Theorem}
\begin{proof}
Following the derivations in Theorem~\ref{Thm:tri_ineq} and Lemma~\ref{Lemma:fournier}, the true probability distribution $\Pcl$ lies within the ambiguity set $\mathcal{B}^{\varepsilon}\left(\Pemp\right)$ of radius $\varepsilon$ with confindence at least $1-\eta$, i.e.,
\begin{equation*}
		\mathbb{P}^N\left\{\Pcl \in \mathcal{B}^{\varepsilon}\left(\Pemp\right) \right\} \geq 1-\eta \ .
\end{equation*}
The proof then follows from the definition of worst-case cost and constraint.
\end{proof}

A benefit of the distributionally robust formulation~\eqref{eq:DR_problem} is that once we get a feasible solution we can directly control the out-of-sample performance. The actual closed-loop cost is guaranteed to be upper bounded by the DR optimal objective with high probability; similarly, feasibility is also guaranteed with high probability. Of course this requires a principled way of choosing the radius $\epsilon$ to guarantee that the true distribution is captured by the ambiguity set. 

On the flip side, the term $\kappa(\eta,N)$ decays slowly, that is, the rate $O(N^{-1/nT})$ is exponentially slow with the system dimension $n$ and horizon $T$. In the literature, this limitation is commonly referred to as the curse of dimensionality~\cite{mohajerin2018data}. In practice, we could use a $\kappa$ different from the one proscribed by Lemma~\ref{Lemma:fournier}. In this case, we can treat $\kappa$ as a hyperparameter and tune it via cross-validation.


\section{Tractable reformulation}\label{Sec:Reformulation}
In this section, we use the result of Theorems~\ref{Thm:tri_ineq},~\ref{Thm:worst-case_guarantees} to obtain a tractable reformulation that approximates problem~\eqref{eq:DR_problem}. We focus on the class of piece-wise affine cost and constraint functions. In particular, we consider cost functions of the form 
\begin{equation}\label{eq:cost_function}
h\left(\bs{y} \right) \!=\! \max_{j\leq N_J} \{h_j(\bs{y})\triangleq\! a_{j} \bs{y} + b_j\},
\end{equation}
for some $N_J>0$. The constraint function is defined similarly
\begin{equation}\label{eq:constraint_function}
g\left(\bs{y}\right) \!=\! \max_{l\leq N_L} \{g_l(\bs{y})\triangleq c_{l} \bs{y} + d_l\},
\end{equation}
for some $N_L\ge 0$.
We argue that the above functions describe rich cost and constraint function classes, including $\ell_1$-norm objectives, e.g. $\snorm{\bf{y}}$. Dealing with other function classes, such as quadratic, would require changing the type ambiguity set (type-$2$ Wasserstein distance, e.g.~\cite{aolaritei2023wasserstein, shafieezadeh2023new}), and lead to a more complex reformulation in presence of model mismatch. We leave that for future work.

As observed from Theorem~\ref{Thm:tri_ineq}, the distance between the predictive and actual closed-loop distributions depends on the decision variable $\Phi$. Moreover, the model uncertainty further complicates this coupling, inducing nonlinearities. To deal with the latter, we appeal to small-gain techniques inspired by robust control and recent advances in robust SLS~\cite{anderson2019system}. In particular, we impose a small-gain condition on the maximum allowed magnitude of the system responses $\Phi$, with the gain scaling inversely proportional to the model error. We control the gain using a hyperparameter $\gamma>0$, over which we optimize.

\begin{Lemma}[Small-gain bound]\label{Lemma:small_gain}
   Assume that $ d_{W}\!\left(\Pwbar, \Pw\right)\le \kappa$, for some $\kappa>0$, pick a $\gamma \in [0,1)$, and let $\Phi$ satisfy
    \begin{equation}\label{eq:constraint_on_gamma}
	\max\{e_A, e_B\}\left\|\PHIh Z\right\|<\gamma \ .
	\end{equation}
    Then,
        $$d_{W}\!\left(\Pemp, \Pcl\right) \leq  \ \varepsilon(\gamma,\Phi,\kappa)$$
		$$\varepsilon(\gamma,\Phi,\kappa) \triangleqq \frac{\gamma}{1-\gamma} \frac{1}{N}\sum_{i=1}^N \left\| \PHIh \wemp -\xui \right\|+ \frac{\kappa}{1-\gamma}\left\| \PHIh Z\right\| \ . $$
\end{Lemma}

The simplified upper bound on the radius $\varepsilon$ is now a convex function of the nominal system responses. Hence, we can now use tools from DR optimization to reformulate~\eqref{eq:DR_problem}. The bound is still non-convex in the auxiliary variable $\gamma$. Since, however, this variable is a scalar, we can perform a grid search over $\gamma$, noting that the bound is a convex function on $\Phi$, for any fixed $\gamma$. 
Note that we use an upper-bound for the true ambiguity radius. As a result, the reformulation is only a conservative approximation of problem~\eqref{eq:DR_problem}. 

\begin{Theorem}[Approximate Reformulation]\label{Thm:Tractabe_worst_case}
    Consider the cost and constraint functions as defined in~\eqref{eq:cost_function},~\eqref{eq:constraint_function} respectively, and let Assumption~\ref{Ass:Light_Tail} hold. 
    Fix a failure probability $\eta>0$ and select $\kappa=\kappa(\eta,N)$ as in~\eqref{eq:kappa_confidence}. Define $\quad\underline{\lambda} = \max_{j\leq  N_J} \left\|a_{j}\right\|_\infty$, $\underline{\theta} = \max_{l\leq N_L} \left\|c_{l}\right\|_\infty$, and $\underline{\varepsilon} = \varepsilon(\gamma,\Phi,\kappa)$ as defined in Lemma~\ref{Lemma:small_gain}.
Consider the \textit{doubly robust} (RR) problem 
\begin{equation}\label{eq:JRR_problem}
		\begin{aligned}
			J^{RR}(\pi)=
			&\inf_{\Phi, \gamma, s_{i}, q_i, t} \underline{\lambda}\, \underline{\varepsilon}+\frac{1}{N} \sum_{i=1}^{N} s_{i} \\
			\text{ s.t. } 
                & \eqref{eq:causality},\ \eqref{eq:SLS_dynamics_constraint_nominal},\ \eqref{eq:constraint_on_gamma},\\
			\text{Cost:}& \left\{\begin{aligned}& a_{j} \hat{\bs{y}}^i + b_{j} \leq s_{i} \\ 
			&{\forall\  i\!=\!1,\dots,N,\  j\!=\!1,\dots, N_J}\  \end{aligned}\right.\\
\text{CVaR:}& \left\{\begin{aligned}&\underline{\theta}\, \underline{\varepsilon} + \frac{1}{N}\! \sum_{i=1}^{N} \!q_{i} \leq t \beta \\
				&\left(c_{l} \hat{\bs{y}}^{i} + d_{l} +t \right)_{+} \leq q_{i} \\
				&\forall\  i=1,\dots,N,\  l=1,\dots, N_L
    \end{aligned}\right.
		\end{aligned}
	\end{equation}
If the problem is feasible, then with probability at least $1-\eta$, \textbf{i)} the optimal cost obtained is an upper bound on the original cost in ~\eqref{eq:original_problem_phi}, and \textbf{ii)} the resulting feedback policy satisfies the CVaR constraint of problem~\eqref{eq:original_problem_phi}. 
\end{Theorem}

For every fixed value of $\gamma$, Problem~\eqref{eq:JRR_problem} is a Linear Programm, similar to the SAA problem~\eqref{eq:problem_SAA} under piece-wise linear convex costs~\eqref{eq:cost_function} and constraints~\eqref{eq:constraint_function}. From this perspective, Problem~\eqref{eq:JRR_problem} belongs to the same complexity class as the SAA optimization problem, while providing ``doubly robust" solutions against the model mismatch and the uncertainty in the disturbance distribution. 

The $\ell_1$ induced norm of the system responses is regularized via the term $\kappa(1-\gamma)^{-1}\|\Phi Z\|$ in $\underline{\epsilon}$ appearing in the cost and constraints.
The penalty coefficient $\kappa$ scales proportionally to the distance between the empirical and true distributions of the disturbance $d_{W}\!\left(\Pwbar, \Pw\right)$. Note that restricting the $\ell_1$ induced norm of the responses has the interpretation of imposing $\ell_1\rightarrow\ell_1$-robustness akin to the $\ell_{\infty}\rightarrow\ell_{\infty}$ robustness in~\cite{anderson2019system}. Unlike standard robust control, the degree of robustness is controlled by the distance $d_{W}\!\left(\Pwbar, \Pw\right)$. The more collected data we have, the milder the regularization.

Constraint~\eqref{eq:constraint_on_gamma} explicitly restricts the norm of the responses, accounting for the effect of model mismatch. This constraint scales inversely with the model errors; smaller errors allow more aggressive controllers. The optimization variable $\gamma$ can be interpreted as a hyperparameter that balances the trade-off between allowing more aggressive controllers in~\eqref{eq:constraint_on_gamma} and suffering from the worst-case distribution shift resulting from the model mismatch as captured by $(1-\gamma)^{-1}$.

Finally, in $\underline{\varepsilon}$, we also penalize differences $\PHIh \wemp -\begin{bmatrix} \bs x^{i,\top}&\bs u^{i,\top} \end{bmatrix}^\top$ between the predicted trajectory and the collected data. This prevents the predicted trajectory from deviating too much from the collected ones, thus, ameliorating the distribution shift due to updating the closed-loop controller. 

When the model error $\Delta\!\mathcal{M}$ is zero, regularizing the system responses robustifies the controller against distribution shifts in the disturbance distribution.
Conversely, if there is no uncertainty about the disturbance distribution, the small gain constraint and the regularization robustify the controller against uncertain dynamics.

We remark that the tractable formulation provided in Problem~\eqref{eq:JRR_problem} is a conservative reformulation of Problem~\eqref{eq:DR_problem}. This is due to the sub-optimalities introduced in the derivations leading to the distribution shift bound in Thm.~\ref{Thm:tri_ineq}, and due to the fact that we are assuming unbounded support for the disturbance distribution which in turns makes the CVaR reformulation conservative.

\section{Extensions}\label{Sec:Extensions}
\subsection{Arbitrary initial conditions}
We address here the more general case where we allow for arbitrary initial conditions in the data collection and control phases. Following the convention of Assumption~\ref{Ass:data}, let \[\bs{w}^i\triangleqq \begin{bmatrix}
        x^i_0 \\
      w^i_{0:T-1}
    \end{bmatrix},i=1,\dots,N\]
    where $x^i_0$ is allowed to vary across different data collection experiments.

We adapt the nominal empirical prediction in~\eqref{eq:prediction_y_i} as follows
$$\hat{\bs{y}}^{i} := \PHIh (\wemp + \tilde{\bs{x}}^i_0) 
$$ with $$\tilde{\bs{x}}^i_0 \triangleqq \begin{bmatrix}
x_0- {x}_0^i\\ 0_{nT\times 1}
\end{bmatrix},\ \ $$
${x}_0^i$ the initial condition of the $i^{th}$ trajectory in the dataset and ${x}_0$ the new initial condition for the control task.
The resulting empirical predictive distribution is defined in the same way as in eq.~\eqref{eq:predictive_empirical_distribution}.

Similarly, we can write the  empirical (finite-sample) version of the
true closed-loop distribution, for a new initial condition $\bar{x}_0$ as in~\eqref{eq:empirical_closed_loop_distribution} but with
$$\yclemp :=  \resolvent \Phi (\wtremp + \tilde{\bs{x}}^i_0)\ .$$ 

Following a similar derivation as in Section~\ref{Sec:Shift}, we can decompose the distance using the triangle inequality~\eqref{eq:Triangle_inequality}.
The component related to the model mismatch can be upper-bounded following the same procedure as in Lemma~\ref{Lem:1} as follows:
\begin{equation*}\scalemath{0.88}{
    \begin{aligned}
            &d_{W}\!\left(\Pemp, \Pclemp \right)
            \leq \frac{1}{N}\sum_{i=1}^N \left\| \resolvent \Phi Z \Delta\!\mathcal{M} \left(\PHIh (\wemp+\tilde{\bs{x}}^i_0) \!-\! \begin{bmatrix}
		    \bs{x}^i \\ \bs{u}^i
		\end{bmatrix} \right) \right\|
\end{aligned}}
\end{equation*}

The bound on the component related to the disturbance distribution uncertainty is unaffected by the new initial condition. This is clear by noting that the first entry of the vectors $\tilde{\bs{w}}$ and ${\bs{w}}^i+\tilde{\bs{x}}^i_0$ is the same and equal to the known new initial condition for the control task $x_0$. With a slight abuse of notation let
 $$ \Pwbar:= \frac{1}{N} \sum_{i=1}^{N} \delta_{{\bs{w}}^{i}+\tilde{\bs{x}}^i_0}\ .$$ 
Then, we recover the same bound as in Lemma~\ref{Lem:2}:
\begin{equation*}
		d_{W}\!\left(\Pclemp, \Pcl\right) \leq \left\| \resolvent \Phi \right\| d_{W}\!\left(\Pwbar, \Pw\right)\ . 
\end{equation*}

Following the derivations in Section~\ref{Sec:Reformulation}, we can formulate the small-gain bound for arbitrary initial conditions as follows.
\begin{Lemma}[Small-gain bound for arbitrary initial conditions]\label{Lemma:small_gain_for_arbitrary_initial_conditions}
   Assume that $ d_{W}\!\left(\Pwbar, \Pw\right)\le \kappa$, for some $\kappa>0$, pick a $\gamma \in [0,1)$, and let $\Phi$ satisfy
    \begin{equation*}
	\max\{e_A, e_B\}\left\|\PHIh Z\right\|<\gamma \ .
	\end{equation*}
    Then,
        \begin{equation*}
            \begin{aligned}
                &\quad \quad d_{W}\!\left(\Pemp, \Pcl\right) \leq \varepsilon(\gamma,\Phi,\kappa) \triangleqq\\
                & \frac{\gamma}{1-\gamma} \frac{1}{N}\sum_{i=1}^N \left\| \PHIh (\wemp+\tilde{\bs{x}}^i_0) - \begin{bmatrix}
		    \bs{x}^i\\ \bs{u}^i
		\end{bmatrix} \right\|+ \frac{\kappa}{1-\gamma}\left\| \PHIh Z\right\| \ .
            \end{aligned}
        \end{equation*}
\end{Lemma}
With the arbitrary initial condition, we maintain the same interpretation of the bound on the ambiguity set radius as in the case with fixed initial condition.

\subsection{Affine SLS formulation}

The results presented in this paper can be naturally extend to the affine system level parametrization formulation. Allowing for a disturbance-affine feedback can be useful for tracking tasks and it can be employed to derive tube-based model predictive control formulations~\cite{sieber2021system}. 
Whenever the initial condition is not zero, the state-feedback policy $\bs{u} = \mathcal{K} \bs{x}$ is already equivalent to an affine feedback policy. It is possible, see e.g.~\cite{sieber2021system}, to introduce an explicit affine term that does not rely on the initial condition being non-zero, we can augment the dynamics to accommodate extended state and disturbance vectors.
The interpretation of the bound remains similar as in the case of linear feedback, but with the extra affine term in the control input. 

\section{Numerical example}\label{Sec:Numerics}
We highlight the need of robustness against model mismatch and finite sample of the disturbance distribution by means of numerical examples. We do that by showing how the doubly robust formulation can handle perturbations in the model and uncertainty related to limited sample sizes much better than the SAA approach. Our results show that the robustness is not detrimental for the performances of the controller even when the model mismatch is not as large as expected, thus making the doubly robust formulation a viable control design option even when no specific robustness guarantees are required.

We consider the system 
\begin{equation*}\label{eq:Dynamics_example}
    \bs{x}_{k+1}=
    \begin{bmatrix}
        0.95 & -0.02\\
        0.0 & 0.2\end{bmatrix} \bs{x}_{k} + 
    \begin{bmatrix}
        0.5 \\
        -0.01 \\
    \end{bmatrix} \bs{u}_{k} +\bs{w}_k,
\end{equation*}
    with additive disturbance $\bs{w}_k \sim \mathcal{N}(0,0.05 I)$, and initial conditions $x_0 = \begin{bmatrix}
        -0.5& -0.5
    \end{bmatrix}^\top$. We consider a horizon $T=10$ and a cost function that regulates the system to the origin 
        \begin{equation*}
            h\left( \xu \right) = \left\|\begin{bmatrix}
                \bs{Q} & \\ & \bs{R}
            \end{bmatrix} \xu \right\| \ ,
        \end{equation*}
        where $\|\cdot\|$ denotes the $\ell_1$ norm,
        with matrices $\bs{Q}$ and $\bs{R}$ block diagonal matrices with blocks $Q=\begin{bmatrix}
                0.01 & 0 \\ 0 & 1
            \end{bmatrix}$ and and $R=\begin{bmatrix}
                0.01
            \end{bmatrix}$ respectively.
    We add a constraint that, at each timestep $k = 1,\dots,T$, constraints the first coordinate of the state to be smaller than $0.8$, i.e.,
    \begin{equation}
        g_k(\bs{y}):\ \ \begin{bmatrix}
            1 & 0
        \end{bmatrix} \bs{x}_{k} - 0.8 \leq 0\ , \quad \forall k=1,\dots,T \ .
    \end{equation}
    This is imposed using the CVaR formulation with $\beta = 0.3$.
    We assume that we have access to a dataset $\mathcal{D}^{N,T}$ comprising $N=20$ trajectories of length $T$. These trajectories have been collected from the system starting from the initial conditions $x_0$ and applying a state-feedback matrix $K=\begin{bmatrix}
    -0.2 & -0.1
    \end{bmatrix}$, i.e. \mbox{$\bs{u}^i_k = K \bs{x}^i_k$}, $k=1,\dots,T$, $i=1,\dots,N$. 
    
We also assume we are given nominal system matrices 
\begin{equation*}
    \hat{A} = \begin{bmatrix} 
            0.95 & 0.01\\
            0.0 & 0.2\end{bmatrix}\ , \quad \hat{B} = \begin{bmatrix}
            0.5 \\
            0.02 \\
        \end{bmatrix}
\end{equation*}
resulting in mismatches $\epsilon_A = \epsilon_B = 0.03$. While we assume that the values of $\epsilon_A$ and $\epsilon_B$ are known, the true dynamics remain unknown. This reflects the practical situations where estimates of the system matrices are obtained through identification, with (often statistical) bounds on the errors. While we consider here the bound on the model error to be known and deterministic, probabilistic bounds can be easily integrated, see for example~\cite{micheli2022data}.
In all the simulations we fix the value of $\kappa=0.005$, this parameter needs to be tuned in practice, for example via cross-validation, see e.g.~\cite{micheli2022data}. We are solving the problem for multiple fixed values of $\gamma \in (0,\ 1)$ and pick the solution that results in the lowest robust optimization cost $J^{RR}$ of Problem~\ref{eq:JRR_problem}.

We first compare the optimal solutions of the RR and SAA approaches. In Fig.~\ref{Fig:planned_trajectories} we compare the predicted optimal trajectories for both algorithms. We can observe that the SAA algorithm plans much more aggressive trajectories. In Fig.~\ref{Fig:validation_trajectories} we show the closed-loop trajectories produced by the respective controller on the true system for $100$ new realization of the random disturbance vector.

\begin{figure}[t]
\includegraphics[width=1\columnwidth]{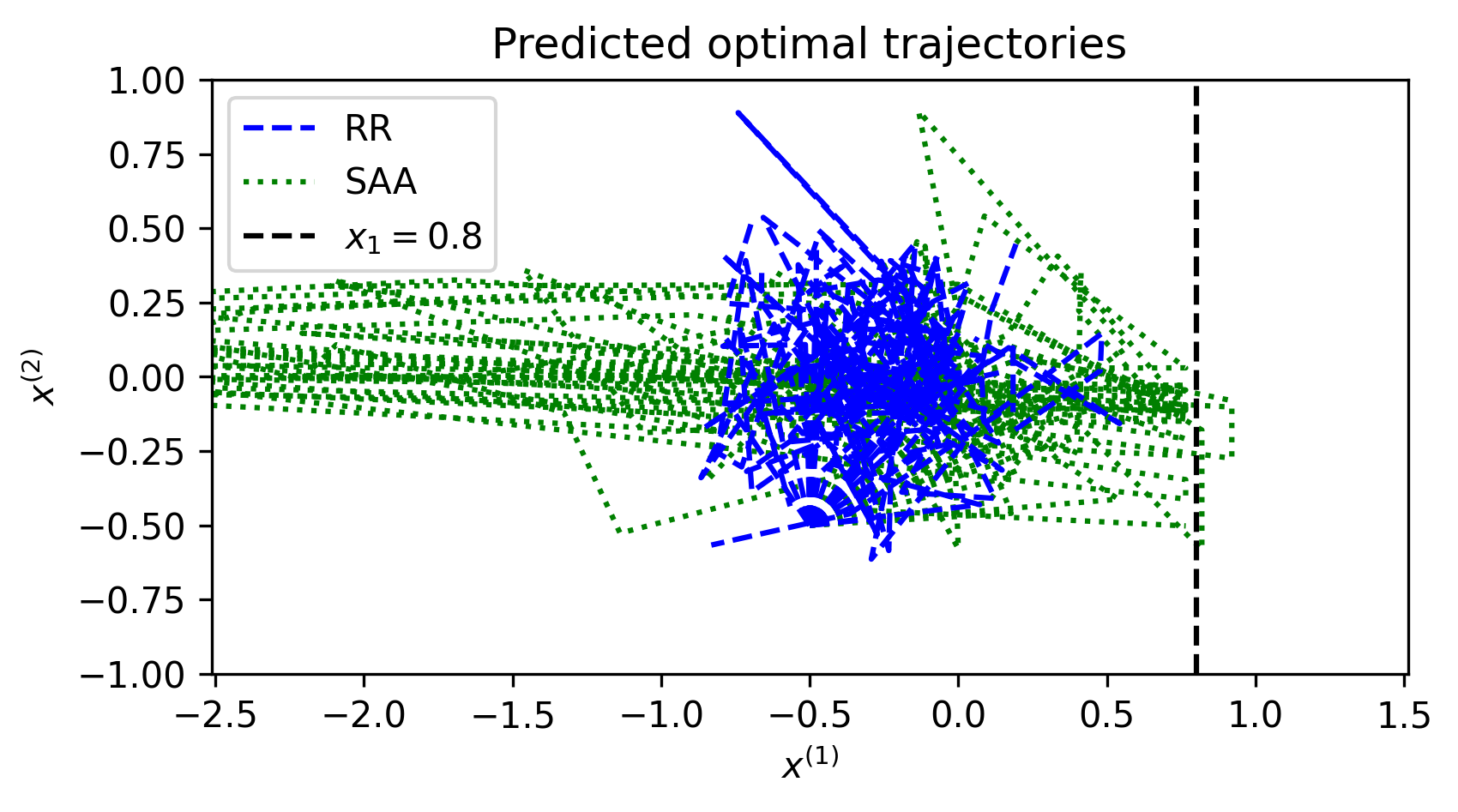}
    \caption{Comparison of optimal predicted trajectories for the RR (blue) and SAA (green) approaches. The vertical black line represents the constraint on the first coordinate of the state.}\label{Fig:planned_trajectories}
\end{figure}

\begin{figure}[t]
\includegraphics[width=1\columnwidth]{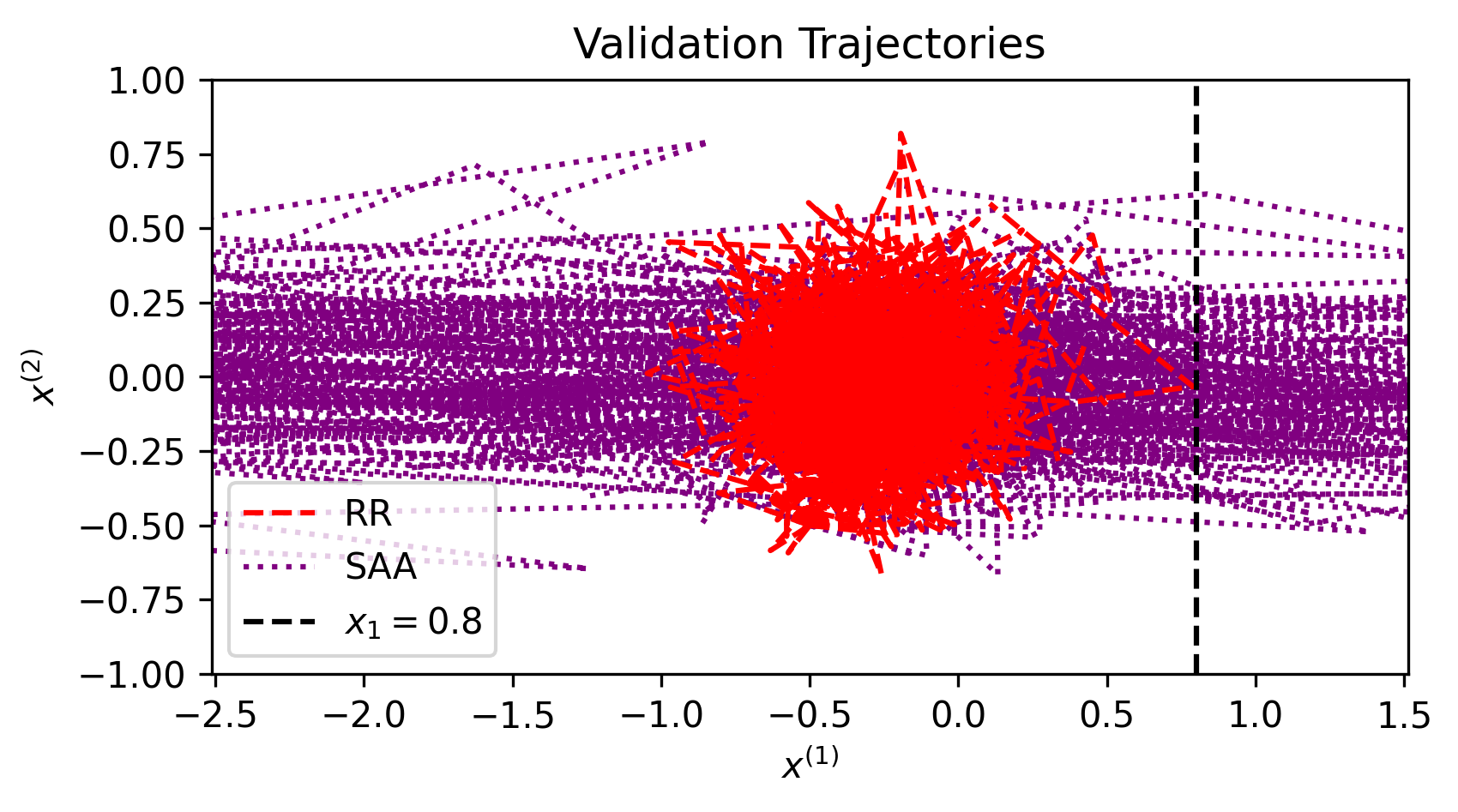}
    \caption{Comparison of validation trajectories for the RR (red) and SAA (purple) approaches. The SAA optimal controller is not robust against the model mismatch and the disturbance realizations, resulting in higher cost and large constraint violations.}\label{Fig:validation_trajectories}
\end{figure}

In the previous example, the model uncertainty severely affects the behavior of the plant dominating the closed-loop performance. That is because the sign of the elements $A_{12}$ and $B_{2}$ of the state and input matrix can be flipped resulting in different behaviors. In the following example, we demonstrate the performance obtained for random model mismatch realizations. We do so by sampling random model mismatch matrices $\Delta A_p$, $\Delta B_p$, $p=1,\dots,50$ that are scaled to obtain an uniform random distribution of model mismatches norms $\| \Delta A_p \|, \| \Delta B_p \| \in \mathcal{U} [0, 0.03]$. For every sample of model mismatch we have an independent dataset of $N = 20$ trajectories collected from the true system and we validate the performance against $100$ validation trajectories. In Fig.~\ref{Fig:boxplot} we can observe the distribution across the $50$ model realizations of the empirical (over the $100$ validation trajectories) validation cost and CVaR values. The CVaR constraint is to be considered violated if it is larger than $0$. 

\begin{figure}[t]
\includegraphics[width=1\columnwidth]{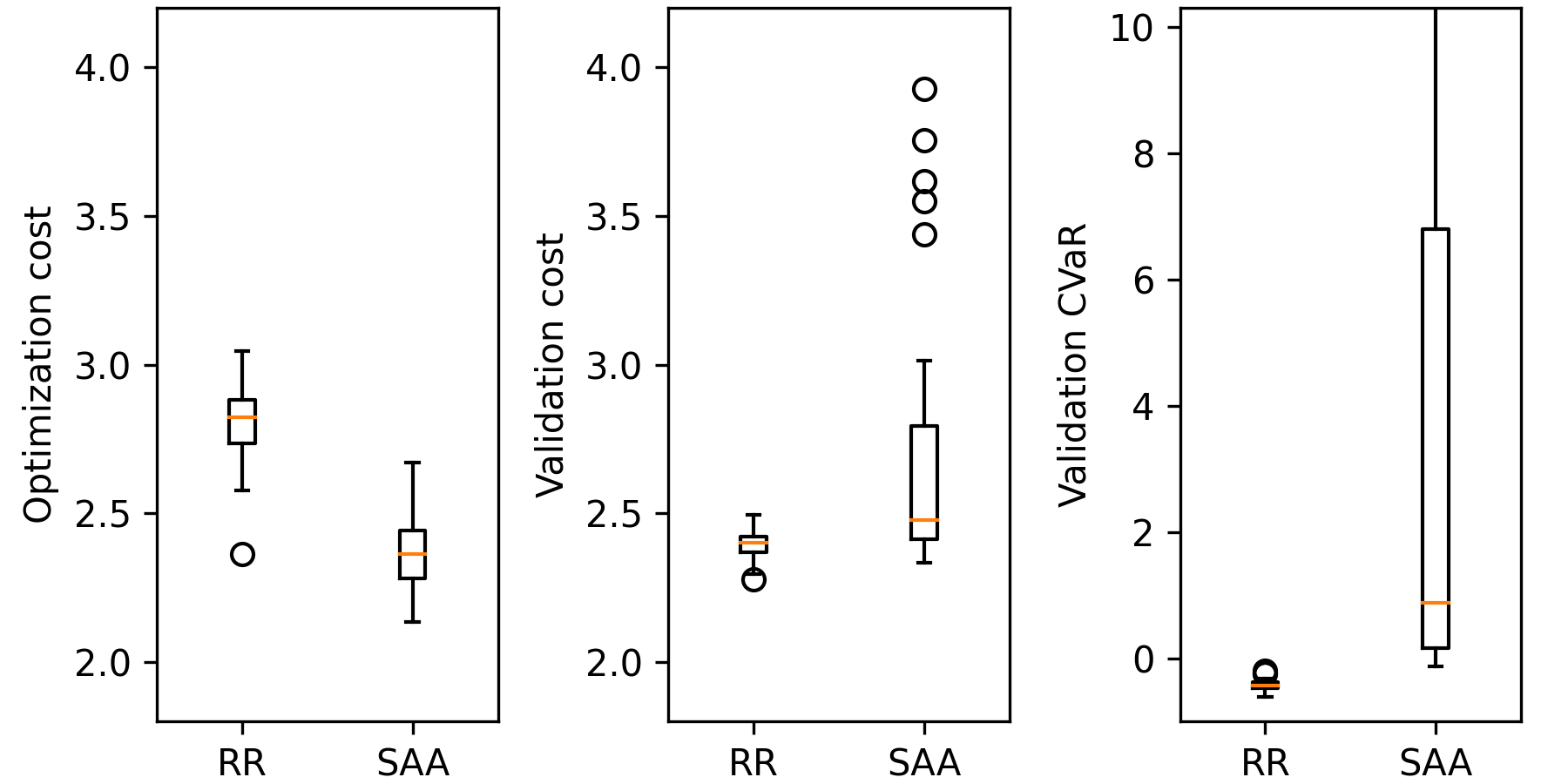}
    \caption{Comparison of optimization cost, validation cost and validation CVaR value, for the RR and SAA approaches.}\label{Fig:boxplot}
\end{figure}

We can observe that, while the optimization costs, i.e. relative to the predicted optimal trajectories, of the SAA are the lowest, the resulting controllers lead to very large validation cost and CVaR values when deployed on the true (unknown) dynamics. Conversely, the RR optimization results in higher optimization costs, that, following Theorem~\ref{Thm:worst-case_guarantees}, provide an upper bound on the validation cost attained on the real system. This fact is corroborated by the validation cost attained by the RR. We can make a similar statement for the CVaR constraint that is consistently violated by the SAA and always satisfied by the RR approach. The RR is therefore able to effectively robustify against the distributional shift induced by both the model mismatch and by the offline dataset limited size. This analysis shows that the RR approach is not too conservative with respect to the SAA even when the model error is not fundamentally altering the plant behavior, while it is always able to maintain robustness against the distribution shift. While some improvement could be obtained for the SAA by increasing the number of samples, which would reduce its sensitivity to the uncertainty in the disturbance distribution, the SAA algorithm does not have a principled way to robustify against the model mismatch.

We remark that, for large values of $\epsilon_A$ and $\epsilon_B$, the robust problem might be infeasible. This fact is worsened by the suboptimalities introduced by the reformulation that can make the constraints harder to satisfy.
A potential solution is to use smaller values for epsilon, e.g. by refining the quality of the available model with further identification experiments, or by collecting more state input trajectories to reduce the uncertainty about the disturbance distribution.

\section{Conclusions}\label{Sec:Conclusion}
We presented a novel distributionally robust state-feedback data-driven controller for uncertain discrete-time linear time-invariant systems affected by unknown additive disturbances. We formulated the problem as a stochastic optimization problem with respect to the worst-case probability distribution within an ambiguity set centered on the empirical nominal predictive distribution. Utilizing tools from robust System Level Synthesis and Distributionally Robust optimization we characterized how the controller affects the distributional shift between the predictive and the closed-loop distributions in the presence of uncertainty about both the dynamics and the disturbance distribution. This allowed to bound the size of the decision-dependent ambiguity set, providing finite-sample probabilistic guarantees on the worst-case expectation and CVaR constraint in the presence of uncertainty about both the dynamics and the disturbance distribution.
We derived a tractable Linear Programming formulation for the DR optimization problem for piece-wise affine cost and constraint functions, and demonstrated through numerical examples the effectiveness of the proposed doubly robust approach against the distributional shift 
which allow to safely control the system without significantly increasing the attained cost, even in presence of model mismatches and very limited information regarding the disturbance distribution.\\
Future work focuses on extending this framework to the episodic setting, where the controller and the model are iteratively updated exploiting the collected data.

\bibliographystyle{IEEEtran}
\bibliography{main.bib}
\appendix
\subsection{Proof of Lemma~\ref{Lem:1}}
    From the definition of the Wasserstein metric we have
    \begin{equation}\label{eq:lemma1}
        \begin{aligned}
            d_{W}\!\left(\Pemp, \Pclemp \right) := & \inf_{\Pi} \int_{\mathcal{Y}^{2}}\left\|\bs{x}_{1}-\bs{x}_{2}\right\| \Pi\left(\mathrm{d} \bs{x}_{1}, \mathrm{d} \bs{x}_{2}\right) \\
            \leq & \frac{1}{N}\sum_{i=1}^N \| \yemp - \yclemp \| \,
        \end{aligned}
    \end{equation}
    where the inequality follows from choosing the (sub-optimal) coupling $\Pi(\bs{x}_{1}=\yemp,\bs{x}_{2}=\yclemp)=1/N$, for all $i=1,\dots,N$.

    We can write $\yemp - \yclemp$ as
    \begin{equation*}
    \begin{aligned}
        \yemp\!-\! \yclemp = & \PHIh \wemp - \resolvent \Phi \wtremp\\
        	\!= & \PHIh \wemp - \resolvent \Phi \left(\wemp + Z \Delta\!\mathcal{M} \xui \right)\\
        	\!= & \resolvent \DELTAtilde \Phi \wemp - \resolvent \DELTAtilde \xui \\
            \!= & \resolvent \DELTAtilde \left(\PHIh \wemp -\xui \right)\ ,
    \end{aligned}
    \end{equation*}
    where the second equality follows from~\eqref{eq:dynamics_matrix_form} and~\eqref{eq:nominal_empirical_distribution_of_the_disturbances}.
Substituting this in~\eqref{eq:lemma1} concludes the proof. \hfill $\blacksquare$

\subsection{Proof of Lemma~\ref{Lem:2}}
From the definition of the Wasserstein metric we can write
 \begin{equation*}
     \begin{aligned}
     & d_{W}\!\left(\Pclemp, \Pcl \right)
        := \inf_{\Pi} \left\{\int_{\mathcal{Y}^2} \left\| {\bs{y}}^\prime - \bs{y}  \right\| \Pi\left(\mathrm{d} {\bs{y}}^\prime, \mathrm{d} \bs{y}\right)\right\}\\
        =& \inf_{\Pi} \left\{\!\int_{\!\mathcal{W}^2} \!\left\| \resolvent \Phi \left( {\bs{w}}^\prime - \bs{w} \right) \right\| \Pi\left(\mathrm{d} {\bs{w}}^\prime, \mathrm{d} \bs{w}\right)\right\}\\
        \le& \inf_{\Pi} \left\{\!\int_{\!\mathcal{W}^2} \!\left\| \resolvent \Phi Z\right\|\left\| \left( {\bs{w}}^\prime - \bs{w} \right) \right\| \Pi\left(\mathrm{d} {\bs{w}}^\prime, \mathrm{d} \bs{w}\right)\right\}\\
        = & \left\| \resolvent \Phi Z \right\| \inf_{\Pi} \left\{\!\int_{\!\mathcal{W}^2} \!\left\| \left( {\bs{w}}^\prime - \bs{w} \right) \right\| \Pi\left(\mathrm{d} {\bs{w}}^\prime, \mathrm{d} \bs{w}\right)\right\}\\   
        = & \left\| \resolvent \Phi Z \right\| d_{W}\!\left(\Pwbar, \Pw \right) \ ,\\ 
        & &\qed
     \end{aligned}
 \end{equation*}
where the inequality follows from the sub-multiplicativity property of the norm and the fact that the first block element of the disturbance $\bs{w}$, i.e. the one related to the initial condition, is deterministic and the same for both vectors. Hence, the first block entry of ${\bs{w}}^\prime - \bs{w}$ always $0$, which allows the first $n$ columns of $R_{\Phi}\Phi$ to be excluded.

\subsection{Proof of Lemma~\ref{Lemma:small_gain}}
The condition~\eqref{eq:constraint_on_gamma} implies that 
\begin{equation*}
    \left\|\DELTAtilde\right\|<\gamma<1\ .
\end{equation*}
This condition allows us to apply the property of convergence of the Neumann series to bound the inverse term as
\begin{equation*}\begin{aligned}
			\| (I-\DELTAtilde)^{-1} \| & \leq \sum_{j=0}^{\infty} (\| \DELTAtilde \|)^{j} \\
			& \leq \frac{1}{1 - \left\| \DELTAtilde \right\|} \leq \frac{1}{1 - \gamma}\ .
		\end{aligned}
	\end{equation*}
We can now upper bound the terms appearing in Theorem~\ref{Thm:tri_ineq} as follows:
	\begin{equation*}
		\scalemath{0.88}{\begin{aligned}
			& d_{W}\left(\Pemp, \Pclemp\right) \\
			\leq & \frac{1}{N}\sum_{i=1}^N \left\| \resolvent \Phi Z \Delta\!\mathcal{M} \left(\PHIh \wemp -\xui \right) \right\| \\
			= & \frac{1}{N}\sum_{i=1}^N \left\| \left(I+\DELTAtilde\right)^{-1} \DELTAtilde \left(\PHIh \wemp -\xui \right) \right\| \\
			\leq & \frac{\gamma}{1-\gamma} \frac{1}{N}\sum_{i=1}^N \left\| \PHIh \wemp -\xui \right\| \ ,
		\end{aligned}}
	\end{equation*}
	and 
	\begin{equation*}
		\begin{aligned}
			d_{W}\left(\Pclemp, \Pcl\right)
			\leq & \left\| \left(I+\Phi Z \Delta\!\mathcal{M}\right)^{-1} \Phi Z\right\| d_{W}\!\left(\Pwbar, \Pw\right)\\
            \leq & \frac{1}{1-\gamma} \left\| \Phi Z\right\| d_{W}\!\left(\Pwbar, \Pw\right)\\
			\leq & \frac{\kappa}{1-\gamma}\, \left\| \Phi Z\right\| \ ,
		\end{aligned}
	\end{equation*}
where we exploited the sub-multiplicativity property of the norm.
\qed

 \subsection{Proof of Theorem~\ref{Thm:Tractabe_worst_case}}
     For a convex piecewise affine cost function $h\left(\bs{y}\right)$, we can apply the results from Theorem 6.3 and Remark 6.6 of~\cite{mohajerin2018data} to reformulate the supremum appearing in the worst-case expectation of~\eqref{eq:DR_problem} as
    
\begin{equation*}
    \begin{aligned}
        \sup_{\mathbb{Q} \in \mathcal{B}^{\varepsilon}\left(\Pemp\right)} \mathbb{E}^{\bs{y}\sim\mathbb{Q}}\left[ h\left(\bs{y} \right) \right] = &
	\min_{\lambda} \ \lambda\, \underline{\varepsilon} +\frac{1}{N} \sum_{i=1}^{N} h(\bs{\hat{y}}^i)\\
       & \ \text{s.t.} \quad \|a_j\|_{\infty} \leq \lambda \quad \forall\ j=1,\dots,N_j\ .
\end{aligned}
\end{equation*}
From the definition of CVaR, the constraints in Problem~\eqref{eq:DR_problem} can be written as the set
\begin{equation*}
        \begin{aligned}
            &\left\{ \Phi, \gamma \middle|\, \sup_{\mathbb{Q} \in \mathcal{B}^{\varepsilon}\left(\Pemp\right)} \text{CVaR}_{1-\beta}^{\bs{y}\sim\mathbb{Q}}\left(g(\bs{y})\right) \leq 0 \right\}\\
            = & \left\{  \Phi, \gamma \middle|\, \sup_{\mathbb{Q} \in \mathcal{B}^{\varepsilon}\left(\Pemp\right)} \inf_{t \in \mathbb{R}} \left[\mathbb{E}^{\bs{y}\sim\mathbb{Q}}\left[(g\left(\bs{y} \right)+t)_{+}\right]-t \beta\right] \leq 0 \right\}\\
            \supseteq & \left\{  \Phi, \gamma \middle|\, \inf_{t \in \mathbb{R}} \sup_{\mathbb{Q} \in \mathcal{B}^{\varepsilon}\left(\Pemp\right)} \left[\mathbb{E}^{\bs{y}\sim\mathbb{Q}}\left[(g\left(\bs{y} \right)+t)_{+}\right]-t \beta\right] \leq 0 \right\}
    \end{aligned}
\end{equation*}
Noting that $(g\left(\bs{y} \right)+t)_+$ is a convex piecewise affine function for $g\left(\bs{y} \right)$ convex and piecewise affine, we can follow the same procedure and rewrite the worst-case as
\begin{equation*}
    \begin{aligned}
        \sup_{\mathbb{Q} \in \mathcal{B}^{\varepsilon}\left(\Pemp\right)} \mathbb{E}^{\bs{y}\sim\mathbb{Q}}&\left[ (g\left(\bs{y}\right)+t)_+ \right] =\\
        =&\quad \min_{\theta} \ \theta\, \underline{\varepsilon} +\frac{1}{N} \sum_{i=1}^{N} (g(\bs{\hat{y}}^i)+t)_+\\
       &\quad \ \text{s.t.} \quad \|c_l\|_{\infty} \leq \theta \quad \forall\ l=1,\dots,N_L\ .
\end{aligned}
\end{equation*}
At optimality the minimum over $\lambda$ and $\theta$ is obtained for $\lambda = \underline{\lambda} = \max_{j = 1, \dots, N_J} \|a_j\|_{\infty}$ and $\theta = \underline{\theta} = \max_{l = 1, \dots, N_L} \|a_l\|_{\infty}$.
Taking the infimum over $\Phi$ for any fixed $\gamma$ results in the finite-dimensional convex program
\begin{equation*}
    \begin{aligned}{J}^{RR}(\pi) =
        \inf _{\Phi, s_{i},q_i}\ & \underline{\lambda}\,  \underline{\varepsilon} +\frac{1}{N} \sum_{i=1}^{N} s_{i} \\
        \text { s.t. } & \eqref{eq:causality},\  \eqref{eq:SLS_dynamics_constraint_nominal}, \ \eqref{eq:constraint_on_gamma}, \\ 
        & a_{j}\hat{\bs{y}}^i + b_{j} \leq s_{i} \\
        & \quad{\forall\  i\!=\!1,\dots,N,\  j\!=\!1,\dots,N_J} \\
        \text{CVaR:}& \left\{\begin{aligned}&\underline{\theta}\, \underline{\varepsilon} + \frac{1}{N}\! \sum_{i=1}^{N} \!\left(c_{l} \hat{\bs{y}}^{i} + d_{l} +t \right)_{+} \leq t \beta \\
        &\forall\  i=1,\dots,N,\  l=1,\dots, N_L
    \end{aligned}\right.
    \end{aligned}\ ,
\end{equation*}

\noindent The claim follows by noting that $\underline{\varepsilon}$ holds uniformly for any~$\Phi$. 
\qed

\begin{IEEEbiography}[{\includegraphics[width=1in,height=1.25in,clip,keepaspectratio]{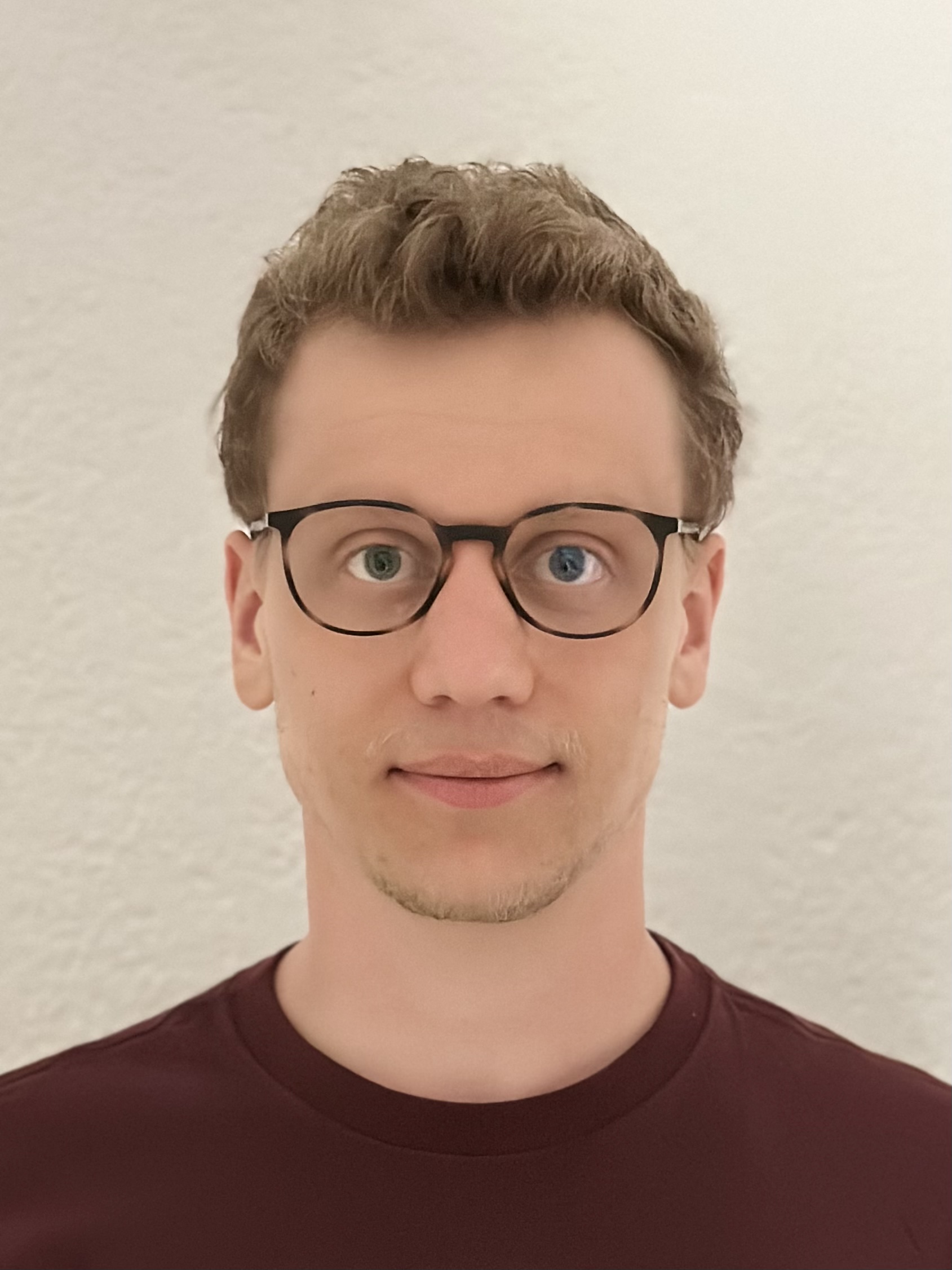}}]{Francesco Micheli} is a PhD student at the Automatic Control Laboratory at ETH Zurich, under the supervision of Prof. J. Lygeros. He received his B.Sc. and M.Sc. degrees in Mechanical Engineering from Politecnico di Milano, Italy, in 2017 and 2018, respectively. His research focuses on safe learning and control, distributionally robust optimization, and robotics.
\end{IEEEbiography}

\begin{IEEEbiography}[{\includegraphics[width=1in,height=1.25in,trim={4cm 12cm 4cm 0},clip,keepaspectratio]{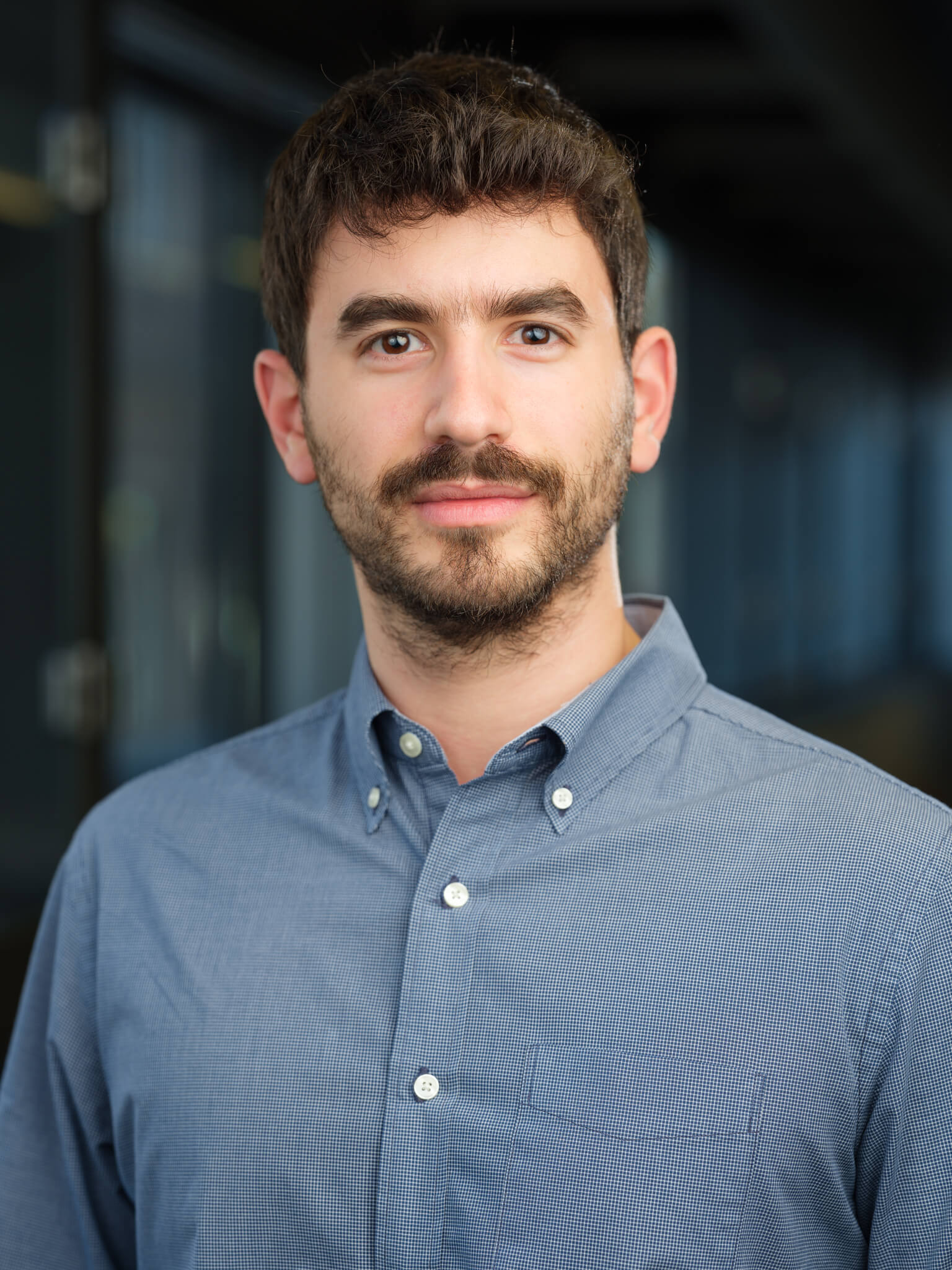}}]{Anastasios Tsiamis} (Member, IEEE) received the Diploma
	degree in electrical and computer engineering from
	the National Technical University of Athens, Greece, in 2014. He obtained his Ph.D. at the Department of Electrical and Systems Engineering, University of Pennsylvania, Philadelphia, PA, USA in 2022. Currently, he is a postdoctoral researcher at the Automatic Control Laboratory, ETH Zurich, Switzerland. His research interests include statistical learning for control, risk-aware control and optimization, and networked control systems.
Anastasios Tsiamis was a finalist for the IFAC Young Author Prize in IFAC 2017 World Congress and a finalist for the Best Student Paper Award in ACC 2019.
\end{IEEEbiography}

\begin{IEEEbiography}[{\includegraphics[width=1in,height=1.25in,clip,keepaspectratio]{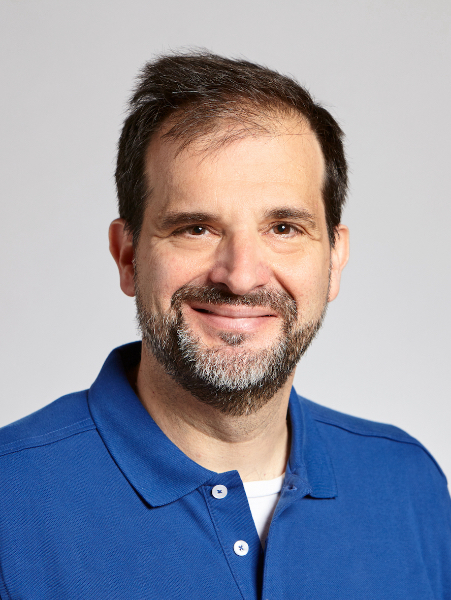}}]{John Lygeros} (Fellow, IEEE) completed a B.Eng. degree in electrical engineering in 1990 and an M.Sc. degree in Systems Control in 1991, both at Imperial College of Science Technology and Medicine, London, U.K.. In 1996 he obtained a Ph.D. degree from the Electrical Engineering and Computer Sciences Department, University of California, Berkeley. During the period 1996-2000 he held research appointments at the National Automated Highway Systems Consortium, Berkeley, the Laboratory for Computer Science, M.I.T., and the Electrical Engineering and Computer Sciences Department at U.C. Berkeley. Between 2000 and 2003 he was a University Lecturer at the Department of Engineering, University of Cambridge, U.K., and a Fellow of Churchill College. Between 2003 and 2006 he was an Assistant Professor at the Department of Electrical and Computer Engineering, University of Patras, Greece. In July 2006 he joined the Automatic Control Laboratory at ETH Zurich, where he is currently serving as the Head of the laboratory. His research interests include modelling, analysis, and control of hierarchical, hybrid, and stochastic systems, with applications to biochemical networks, transportation systems, energy systems, and industrial processes. John Lygeros is a Fellow of the IEEE, and a member of the IET and the Technical Chamber of Greece; between 2013 and 2023 he served as the Vice President for Finances and a Council Member of the International Federation of Automatic Control (IFAC), as well as on the Board of the IFAC Foundation.
\end{IEEEbiography}

\end{document}